\documentclass[final,leqno]{siamltex}
\usepackage{amsmath,amssymb,upgreek,nicefrac,bm,cite,verbatim}
\usepackage[mathscr]{eucal}

\newcommand{\cK}{\mathcal{K}}
\newcommand{\cT}{\mathcal{T}}
\newcommand{\order}{{\mathscr{O}}} 
\newcommand{\VI}{{\overline{\varphi}}}
\newcommand{\RVI}{{\varphi}}
\newcommand{\Or}{{\overline{r}}}
\newcommand{\Equil}{\mathscr{E}}
\newcommand{\tc}{{\Breve\uptau}}
\newcommand{\Hv}{\Hat{v}}
\newcommand{\Bo}{B_{0}}
\DeclareMathOperator*{\Argmin}{Arg\,min}
\DeclareMathOperator*{\osc}{osc}
\newcommand{\cG}{\mathcal{G}}

\DeclareMathOperator{\Exp}{\mathbb{E}}
\DeclareMathOperator{\Prob}{\mathbb{P}}
\newcommand{\D}{\mathrm{d}}
\newcommand{\E}{\mathrm{e}}
\newcommand{\df}{\,\triangleq\,}
\newcommand{\transp}{^{\mathsf{T}}}
\newcommand{\Uadm}{\mathfrak{U}}
\newcommand{\Usm}{\mathfrak{U}_{\mathrm{SM}}}
\newcommand{\Ussm}{\mathfrak{U}_{\mathrm{SSM}}}
\newcommand{\Lg}{L}
\newcommand{\sF}{\mathfrak{F}}
\newcommand{\Ind}{\mathbb{I}}
\newcommand{\Cc}{\mathcal{C}}
\newcommand{\Ccl}{\mathcal{C}_{\mathrm{loc}}}

\newcommand{\Sob}{\mathscr{W}}
\newcommand{\Sobl}{\mathscr{W}_{\mathrm{loc}}}
\newcommand{\Act}{\mathbb{U}}
\newcommand{\imeas}{\mu}
\newcommand{\Lyap}{\mathcal{V}}
\newcommand{\RR}{\mathbb{R}}
\newcommand{\NN}{\mathbb{N}}
\newcommand{\abs}[1]{\lvert#1\rvert}
\newcommand{\babs}[1]{\bigl\lvert#1\bigr\rvert}
\newcommand{\norm}[1]{\lVert#1\rVert}

\newtheorem{remark}[theorem]{{\it Remark}}
\newtheorem{assumption}[theorem]{{\it Assumption}}

\title
{Convergence of The Relative Value Iteration for\\
the Ergodic Control Problem of Nondegenerate\\
Diffusions under Near-Monotone Costs}
\author{Ari Arapostathis\thanks{Department of Electrical and Computer
Engineering, The University of Texas at Austin, 1 University Station,
Austin, TX 78712 (ari@mail.utexas.edu).
This author's work was supported in part by the Office of Naval Research
through the Electric Ship Research and Development Consortium.}
\and Vivek S. Borkar\thanks{Department of Electrical Engineering,
Indian Institute of Technology, Powai, Mumbai 400076, India.
(borkar.vs@gmail.com). This author's work was supported in part by the 
J.~C.~Bose Fellowship from the Government of India
and grant 11IRCCSG014 from IRCC, IIT Bombay.}
\and K. Suresh Kumar\thanks{Department of Mathematics,
Indian Institute of Technology, Powai, Mumbai 400076, India.
(suresh@math.iitb.ac.in).}
}

\begin{document}
\maketitle

\begin{abstract}
We study the relative value iteration for the ergodic control problem
under a near-monotone running cost structure
for a nondegenerate diffusion controlled through its drift.
This algorithm takes the form of a quasilinear parabolic Cauchy initial value
problem in $\RR^{d}$.
We show that this Cauchy problem stabilizes,
or in other words, that the solution of the quasilinear parabolic equation
converges for every bounded initial condition in $\Cc^{2}(\RR^{d})$ to
the solution of the Hamilton--Jacobi--Bellman (HJB) equation associated
with the ergodic control problem.
\end{abstract}



\pagestyle{myheadings}
\thispagestyle{plain}
\markboth{ARI ARAPOSTATHIS, VIVEK S. BORKAR AND K. SURESH KUMAR}
{THE RELATIVE VALUE ITERATION UNDER NEAR-MONOTONE COST}

\section{Introduction}\label{S1}

This paper is concerned with the time-asymptotic behavior of
an optimal control problem for a nondegenerate diffusion
controlled through its drift and described by an It\^o stochastic
differential equation (SDE) in $\RR^{d}$ having the following form:
\begin{equation}\label{E-sde}
\D{X}_{t} = b(X_{t},U_{t})\,\D{t} + \upsigma(X_{t})\,\D{W}_{t}\,.
\end{equation}
Here $U_{t}$ is the control variable that takes values in some compact metric
space.
We impose standard assumptions on the data to guarantee the existence and
uniqueness of solutions to \eqref{E-sde}.
These are described in \S\ref{S-model}.
Let $r\colon \RR^{d} \times\Act\to\RR$ be a continuous function
bounded from below, which without loss of generality we assume it is nonnegative,
referred to as the \emph{running cost}.
As is well known, the ergodic control problem, in its \emph{almost sure}
(or \emph{pathwise}) formulation,
seeks to a.s.\ minimize over all admissible controls $U$ the functional
\begin{equation}\label{E-ergcrit}
\limsup_{t\to\infty}\; \frac{1}{t}\int_{0}^{t} r(X_{s},U_{s})\,\D{s}\,.
\end{equation}
A weaker, \emph{average} formulation seeks to minimize
\begin{equation}\label{E-avgcrit}
\limsup_{t\to\infty}\;\frac{1}{t}\int_{0}^{t}
\Exp^{U}\bigl[r(X_{s},U_{s})\bigr]\,\D{s}\,.
\end{equation}
Here $\Exp^{U}$ denotes the expectation operator associated
with the probability measure on the canonical
space of the process under the control $U$.
We let $\varrho$ be defined as 
\begin{equation}\label{E-varrho}
\varrho\df \inf_{U}\; \limsup_{t\to\infty}\;\frac{1}{t}\int_{0}^{t}
\Exp^{U}\bigl[r(X_{s},U_{s})\bigr]\,\D{s}\,,
\end{equation}
i.e., the infimum of \eqref{E-avgcrit} over all admissible controls
(for the definition of admissible controls see \S\ref{S-model}).
Under suitable hypotheses solutions to the ergodic control problem can be
synthesized via the Hamilton--Jacobi--Bellman (HJB) equation
\begin{equation}\label{E-HJB}
a^{ij}(x)\,\partial_{ij} V + H(x,\nabla V) = \varrho\,,
\end{equation}
where $a =[a^{ij}]$ is the symmetric
matrix $\frac{1}{2}\upsigma\,\upsigma\transp$ and
\begin{equation*}
H(x,p)\df\min_{u}\; \left\{b(x,u)\cdot p + r(x,u)\right\}\,.
\end{equation*}
The desired characterization is that a stationary Markov control $v$ is optimal
for the ergodic control problem if and only if it satisfies
\begin{equation}\label{E-v*}
H\bigl(x,\nabla V(x)\bigr)=b\bigl(x,v(x)\bigr)\cdot \nabla V(x)
+ r\bigl(x,v(x)\bigr)
\end{equation}
a.e.\ in $\RR^{d}$.
Obtaining solutions to \eqref{E-HJB} is further complicated by the fact that
$\varrho$ is unknown.
For controlled Markov chains the \emph{relative value iteration}
originating in the work of White \cite{White-63} provides an algorithm for
solving the ergodic
dynamic programming equation for the finite state finite action case.
Moreover its ramifications have given rise to popular learning algorithms
($Q$-learning) \cite{Abounadi-01}.

In \cite{RVI} we introduced a
continuous time, continuous state space analog of White's relative value iteration
(RVI) given by the quasilinear parabolic evolution equation
\begin{equation}\label{E-RVI}
\partial_{t}{\RVI}(t,x) = a^{ij}(x)\,\partial_{ij} {\RVI}(t,x)
+ H(x,\nabla {\RVI}) - {\RVI}(t,0)\,,\quad {\RVI}(0,x) = {\RVI}_{0}(x)\,.
\end{equation}
Under a uniform (geometric) ergodicity
condition that ensures the well-posedness of the associated
HJB equation we showed in \cite{RVI} that the
solution of \eqref{E-RVI} converges as $t\to\infty$
to a solution of \eqref{E-HJB}, the limit being independent of the
initial condition ${\RVI}_{0}$.
In a related work we extended these results to zero-sum stochastic differential
games and controlled diffusions under the risk sensitive criterion
\cite{arXiv-games}.

Even though the work in \cite{RVI} was probably the first such study of convergence
of a relative iteration scheme for continuous time and space Markov processes,
the blanket stability hypothesis imposed weakens these results.
Models of controlled diffusions enjoying a uniform geometric ergodicity
do not arise often in applications.
Rather, what we frequently encounter is a running cost which has a structure
which penalizes unstable behavior and thus renders all stationary optimal
controls stable.
Such is the case for quadratic costs typically used in linear control models.
A fairly general class of running costs of this type,
which includes `norm-like' costs, consists of costs satisfying
the \emph{near-monotone} condition:
\begin{equation}\label{E-nm}
\bigl\{x\in\RR^{d} : \min_{u}\;r(x,u) \le \varrho\bigr\}\ \
\text{is a compact set.}
\end{equation}

In this paper we relax the blanket geometric ergodicity assumption
and study the relative value iteration in \eqref{E-RVI}
under the near-monotone hypothesis \eqref{E-nm}.
It is well known that for near-monotone costs the HJB equation \eqref{E-HJB}
possesses a unique up to a constant solution $V$ which is bounded below
in $\RR^{d}$ \cite{book}.
However, this uniqueness result is restricted.
In general, for $\beta>\varrho$ the equation
\begin{equation}\label{E-HJBbeta}
a^{ij}(x)\,\partial_{ij} V + H(x,\nabla V) = \beta
\end{equation}
can have a multitude of solutions which are bounded below \cite{book}.
As a result, the policy iteration algorithm (PIA) may fail to converge to the
optimal value \cite{Meyn,OneFest}.
In order
to guarantee convergence of the PIA to an optimal control,
in addition to the near-monotone assumption,
a blanket Lyapunov condition is imposed in \cite[Theorem~5.2]{Meyn} which
renders all stationary Markov controls stable.
In contrast, the RVI algorithm always converges to the optimal value function
when initialized with some bounded initial value ${\RVI}_{0}$.
The reason behind the difference in performance of the two algorithms can
be explained as follows:
First, recall that the PIA algorithm consists of the following steps:
\begin{remunerate}
\item
Initialization. 
Set $k=0$ and select some stationary Markov control $v_{0}$
which yields a finite average cost.
\item
Value determination.
Determine the average cost $\varrho_{v_{k}}$ under the control $v_{k}$
and obtain a solution $V_{k}$ to the Poisson equation
\begin{equation*}
a^{ij}(x)\,\partial_{ij} V_{k} 
+ b^{i}\bigl(x,v_{k}(x)\bigr)\partial_{i} V_{k}(x)
+ r\bigl(x,v_{k}(x)\bigr) = \varrho_{v_{k}}\,,\quad x\in\RR^{d}\,.
\end{equation*}
\item
Termination.
If
\begin{equation*}
H(x,\nabla V_{k})
= \bigl[b\bigl(x,v_{k}(x)\bigr)\cdot \nabla V_{k}(x) + r\bigl(x,v_{k}(x)\bigr)\bigr]
\quad\text{a.e.}\,,
\end{equation*}
then return $v_{k}$.
\item
Policy improvement.
Select $v_{k+1}\in\Usm$ which satisfies
\begin{equation*}
v_{k+1}(x)\in
\Argmin_{u\in\Act}\;\bigl[b(x,u)\cdot \nabla V_{k}(x) + r(x,u)\bigr]\,,
\qquad x\in\RR^{d}\,.
\end{equation*}
\end{remunerate}
It is straightforward to show that if $\widehat{V}$ is a solution to
\eqref{E-HJBbeta} whose growth rate does not exceed the growth
rate of an optimal value function $V$ from \eqref{E-HJB},
or in other words the weighted norm
$\norm{\widehat{V}}_{V}$ is finite,
then $\beta=\varrho$ and $\widehat{V}$ is an optimal value function.
It turns out that if
the value function ${\RVI}_{0}$ determined at the first step $k=0$ does not grow
faster than an optimal value function $V$ then the algorithm will converge
to an optimal value function.
Otherwise, it might converge to a solution of \eqref{E-HJBbeta} that is not optimal.
However, the growth rate of an optimal value function is not known,
and there is
no simple way of selecting the initial control $v_{0}$ that will result
in the right growth rate for ${\RVI}_{0}$.
To do so one must solve a HJB-type equation, which is precisely what
the PIA algorithm tries to avoid.
In contrast, as we show in this paper,
the solution of the RVI algorithm has the property
that $x\mapsto{\RVI}(t,x)$ has the same growth rate as the optimal
value function $V$, asymptotically in $t$.
This is an essential ingredient of the mechanism responsible for convergence.

The proof of convergence of \eqref{E-RVI} is facilitated
by the study of the \emph{value iteration} (VI) equation
\begin{equation}\label{E-VI}
\partial_{t}{\VI}(t,x) = a^{ij}(x)\,\partial_{ij} {\VI}(t,x)
+ H(x,\nabla {\VI})
- \varrho\,,\quad {\VI}(0,x)={\RVI}_{0}(x)\,.
\end{equation}
The initial condition is the same as in \eqref{E-RVI}.
Also $\varrho$ is as in \eqref{E-varrho}, so it is assumed known.
Note that if ${\RVI}$ is a solution of \eqref{E-RVI}, then
\begin{equation}\label{E-rvivi}
{\VI}(t,x) = {\RVI}(t,x) -\varrho\,t + \int_{0}^{t} {\RVI}(s,0)\,\D{s}
\,,\qquad (t,x)\in\RR_{+}\times\RR^{d}\,.
\end{equation}
solves \eqref{E-VI}.
We have in particular that
\begin{equation}\label{E-ident}
{\VI}(t,x) - {\VI}(t,0) = {\RVI}(t,x) - {\RVI}(t,0) \qquad \forall x\in\RR^{d}\,,
~~\forall t\ge0\,.
\end{equation}
It follows that the function $f\df{\RVI}-{\VI}$ does not depend on $x\in\RR^{d}$
and satisfies
\begin{equation}\label{E-ode}
\frac{\D f}{\D t} + f = \varrho - \VI(t,0)\,.
\end{equation}
Conversely, if ${\VI}$ is a solution of \eqref{E-VI} then
solving \eqref{E-ode} one obtains a corresponding solution of \eqref{E-RVI}
that takes the form \cite[Lemma~4.4]{RVI}:
\begin{equation}\label{E-virvi}
{\RVI}(t,x) = {\VI}(t,x) - \int_{0}^{t}\E^{s-t}\,{\VI}(s,0)\,\D{s} +
\varrho\,(1-\E^{-t})\,,\qquad (t,x)\in\RR_{+}\times\RR^{d}\,.
\end{equation}
It also follows from \eqref{E-virvi} that if $t\mapsto\VI(t,x)$ is bounded
for each $x\in\RR^{d}$ then so is the map $t\mapsto\RVI(t,x)$, and if
the former converges as $t\to\infty$, pointwise in $x$, then so does
the latter.

We note here that we study solutions of the VI equation
that have the stochastic representation
\begin{equation}\label{E-VISR}
{\VI}(t,x) = \inf_{U}\;\Exp^{U}_{x}
\left[\int_{0}^{t}r(X_{s},U_{s})\,\D{s}+{\RVI}_{0}(X_{t})\right]-\varrho\,t\,,
\end{equation}
where the infimum is over all admissible controls.
These are called \emph{canonical solutions} (see Definition~\ref{D-canonical}).
The first term in \eqref{E-VISR} is the total cost over the
finite horizon $[0,t]$ with terminal penalty ${\RVI}_{0}$.
Under the uniform geometric ergodicity hypothesis used in \cite{RVI}
it is straightforward to show that $t\mapsto\VI(t,x)$ is locally bounded
in $x\in\RR^{d}$.
In contrast, under the near-monotone hypothesis alone, $t\mapsto\VI(t,x)$
may diverge for each $x\in\RR^{d}$.
To show convergence, we first identify a suitable region of attraction of
the solutions of the HJB under the dynamics of \eqref{E-RVI}
and then show that all $\omega$-limit points of the semiflow of
\eqref{E-RVI} lie in this region.

While we prefer to think of \eqref{E-RVI} as a continuous time and space
relative value iteration, it can also be viewed as a `stabilization
of a quasilinear parabolic PDE problem' in analogy to the
celebrated result of Has$'$minski\u{\i}
(see \cite{Hasm-60}).
Thus, the results in this paper
are also likely to be of independent interest to the PDE community.

We summarize below the main result of the paper.
We make one mild assumption: let $v^{*}$ be some optimal stationary
Markov control, i.e., a measurable function that satisfies \eqref{E-v*}.
It is well known that under the near-monotone hypothesis the diffusion
under the control $v^{*}$ is positive recurrent.
Let $\imeas_{v^{*}}$ denote the unique invariant probability measure
of the diffusion under the control $v^{*}$.
We assume that the value function $V$ in the HJB is integrable under
$\imeas_{v^{*}}$.

\begin{theorem}\label{T-main}
Suppose that the running cost is near-monotone and that
the value function $V$ of the HJB equation \eqref{E-HJB} for the ergodic
control problem
is integrable with respect to some optimal invariant probability
distribution.
Then for any bounded initial condition ${\RVI}_{0}\in\Cc^{2}(\RR^{d})$
it holds that
\begin{align*}
\lim_{t\to\infty}\;{\RVI}(t,x) = V(x)-V(0)+\varrho\,,
\end{align*}
uniformly on compact sets of $\RR^{d}$.
\end{theorem}

We also obtain a new stochastic representation for the value function
of the HJB under near-monotone costs which we state as a corollary.
This result is known to hold under uniform geometric ergodicity,
but under the near-monotone cost hypothesis alone it is completely new.

\begin{corollary}
Under the assumptions of Theorem~\ref{T-main}
the value function $V$ of the HJB for the ergodic control problem
has the stochastic representation:
\begin{equation*}
V(x) - V(y)= \lim_{t\to\infty}\;\left(\inf_{U}\;
\Exp_{x}^{U}\left[\int_{0}^{t} r(X_{s},U_{s})\,\D{s}\right]
- \inf_{U}\;\Exp_{y}^{U}
\left[\int_{0}^{t}r(X_{s},U_{s})\,\D{s}\right]\right)
\end{equation*}
for all $x,y\in\RR^{d}$.
\end{corollary}

We would like to note here that in \cite{ChenMeyn-99}
the authors study the value iteration algorithm for
countable state controlled Markov chains, with `norm-like'
running costs, i.e., $\min_{u}\; r(x,u)\to\infty$ as $\abs{x}\to\infty$.
The initial condition ${\RVI}_{0}$ is chosen as
some Lyapunov function corresponding to some stable control $v_{0}$.
We leave it to the reader to verify that under these hypotheses
$\norm{V}_{{\RVI}_{0}}<\infty$.
Moreover they assume that ${\RVI}_{0}$ is integrable with respect to the
invariant probability distribution $\imeas_{v^{*}}$
(see the earlier discussion concerning the PIA algorithm).
Thus their hypotheses imply that the optimal value function $V$
from \eqref{E-HJB} is also
integrable with respect to $\imeas_{v^{*}}$.

The paper is organized as follows.
The next section introduces the notation used in the paper. 
Section~\ref{S-problem} starts by
describing in detail the model and the assumptions imposed.
In \S\ref{S-ergodic} we discuss some basic properties of the HJB equation for the
ergodic control problem under near-monotone costs and the implications
of the integrability of the value function under some optimal invariant
distribution.
In \S\ref{S-VIRVI} we address the issue of existence and
uniqueness of solutions to \eqref{E-RVI} and \eqref{E-VI}
and describe some basic properties of these solutions.
In \S\ref{S-attract} we
exhibit a region of attraction for the solutions of the VI.
In \S\ref{S-growth} we derive some essential growth estimates for the solutions
of the VI and show that these solutions have locally bounded oscillation in $\RR^{d}$,
uniformly in $t\ge0$.
Section~\ref{S-main} is dedicated to the proof of convergence of the solutions
of the RVI,
while \S\ref{S-concl} concludes with some pointers to future work.

\section{Notation}\label{S-not}
The standard Euclidean norm in $\RR^{d}$ is denoted by $\abs{\,\cdot\,}$.
The set of nonnegative real numbers is denoted by $\RR_{+}$,
$\NN$ stands for the set of natural numbers, and $\Ind$ denotes
the indicator function.
We denote by $\uptau(A)$
the \emph{first exit time} of a process
$\{X_{t}\,,\ t\in\RR_{+}\}$ from a set $A\subset\RR^{d}$, defined by
\begin{equation*}
\uptau(A) \df \inf\;\{t>0 : X_{t}\not\in A\}\,.
\end{equation*}
The closure and the boundary of a set $A\subset\RR^{d}$ are denoted
by $\overline{A}$ and $\partial{A}$, respectively.
The open ball of radius $R$ in $\RR^{d}$, centered at the origin,
is denoted by $B_{R}$, and we let $\uptau_{R}\df \uptau(B_{R})$,
and $\tc_{R}\df \uptau(B^{c}_{R})$.

The term \emph{domain} in $\RR^{d}$
refers to a nonempty, connected open subset of the Euclidean space $\RR^{d}$. 
For a domain $D\subset\RR^{d}$,
the space $\Cc^{k}(D)$ ($\Cc^{\infty}(D)$)
refers to the class of all functions whose partial
derivatives up to order $k$ (of any order) exist and are continuous.

We adopt the notation
$\partial_{t}\df\tfrac{\partial}{\partial{t}}$, and for $i,j\in\NN$,
$\partial_{i}\df\tfrac{\partial~}{\partial{x}_{i}}$ and
$\partial_{ij}\df\tfrac{\partial^{2}~}{\partial{x}_{i}\partial{x}_{j}}$.
We often use the standard summation rule that
repeated subscripts and superscripts are summed from $1$ through $d$.
For example,
\begin{equation*}
a^{ij}\partial_{ij}\varphi
+ b^{i} \partial_{i}\varphi \df \sum_{i,j=1}^{d}a^{ij}
\frac{\partial^{2}\varphi}{\partial{x}_{i}\partial{x}_{j}}
+\sum_{i=1}^{d} b^{i} \frac{\partial\varphi}{\partial{x}_{i}}\,.
\end{equation*}

For a nonnegative multi-index $\alpha=(\alpha_{1},\dotsc,\alpha_{d})$
we let $D^{\alpha}\df \partial_{1}^{\alpha_{1}}\dotsb
\partial_{d}^{\alpha_{d}}$.
Let $Q$ be a domain in $\RR_{+}\times\RR^{d}$.
Recall that $\Cc^{r,k+2r}(Q)$ stands for the set of bounded
continuous functions $\varphi(t,x)$ defined on $Q$ such
that the derivatives $D^{\alpha}\partial_{t}^{\ell}\varphi$
are bounded and continuous in $Q$ for
\begin{equation*}
\abs{\alpha}+2\ell \le k+2r\,,\qquad \ell\le r\,.
\end{equation*}

In general if $\mathcal{X}$ is a space of real-valued functions on $Q$,
$\mathcal{X}_{\mathrm{loc}}$ consists of all functions $f$ such that
$f\varphi\in\mathcal{X}$ for every $\varphi\in\Cc_{c}^{\infty}(Q)$,
the space of smooth functions on $Q$ with compact support.
In this manner we obtain for example
the spaces $\Ccl^{1,2}(\RR^{d})$ and $\Sobl^{2,p}(Q)$.

We won't introduce here the parabolic Sobolev space
$\Sob^{r,k+2r,p}(Q)$, since the solutions of \eqref{E-RVI}
and \eqref{E-VI} are in $\Ccl^{1,2}(\RR^{d})$.
The only exception is the function $\psi$
in Theorem~\ref{T-Harnack} and the function $\psi_{T}$ used in the
proof of Lemma~\ref{L-Harn1}.
We refer the reader to \cite{Krylov-08} for definitions and
properties of the parabolic Sobolev space.

\section{Problem Statement and Preliminary Results}\label{S-problem}

\subsection{The model}\label{S-model}

The dynamics are modeled by a
controlled diffusion process $X = \{X_{t},\;t\ge0\}$
taking values in the $d$-dimensional Euclidean space $\RR^{d}$, and
governed by the It\^o stochastic differential equation in \eqref{E-sde}.
All random processes in \eqref{E-sde} live in a complete
probability space $(\Omega, \sF,\Prob)$.
The process $W$ is a $d$-dimensional standard Wiener process independent
of the initial condition $X_{0}$.
The control process $U$ takes values in a compact, metrizable set $\Act$, and
$U_{t}(\omega)$ is jointly measurable in
$(t, \omega)\in[0,\infty)\times\Omega$.
Moreover, it is \emph{non-anticipative}:
for $s < t$, $W_{t} - W_{s}$ is independent of
\begin{equation*}
\sF_{s} \df \text{the completion of~} \sigma\{X_{0}, U_{r}, W_{r},\; r\le s\}
\text{~relative to~} (\sF,\Prob)\,.
\end{equation*}
Such a process $U$ is called an \emph{admissible control},
and we let $\Uadm$ denote the set of all admissible controls.

We impose the following standard assumptions on the drift $b$
and the diffusion matrix $\upsigma$
to guarantee existence and uniqueness of solutions to \eqref{E-sde}.
\begin{description}
\item[(A1)]
\textit{Local Lipschitz continuity:\/}
The functions
\begin{equation*}
b=\bigl[b^{1},\dotsc,b^{d}\bigr]\transp :\RR^{d} \times\Act\mapsto\RR^{d}
\quad\text{and}\quad
\upsigma=\bigl[\upsigma^{ij}\bigr]:\RR^{d}\mapsto\RR^{d\times d}
\end{equation*}
are locally Lipschitz in $x$ with a Lipschitz constant $\kappa_{R}>0$ depending on
$R>0$.
In other words,
for all $x,y\in B_{R}$ and $u\in\Act$,
\begin{equation*}
\abs{b(x,u) - b(y,u)} + \norm{\upsigma(x) - \upsigma(y)}
\le \kappa_{R}\abs{x-y}\,.
\end{equation*}
\item[(A2)]
\textit{Affine growth condition:\/}
$b$ and $\upsigma$ satisfy a global growth condition of the form
\begin{equation*}
\abs{b(x,u)}^{2}+ \norm{\upsigma(x)}^{2}\le \kappa_{1}
\bigl(1 + \abs{x}^{2}\bigr) \qquad \forall (x,u)\in\RR^{d}\times\Act\,,
\end{equation*}
where $\norm{\upsigma}^{2}\df
\mathrm{trace}\left(\upsigma\upsigma\transp\right)$.
\item[(A3)]
\textit{Local nondegeneracy:\/}
For each $R>0$, we have
\begin{equation*}
\sum_{i,j=1}^{d} a^{ij}(x)\xi_{i}\xi_{j}
\ge\kappa^{-1}_{R}\abs{\xi}^{2} \qquad\forall x\in B_{R}\,,
\end{equation*}
for all $\xi=(\xi_{1},\dotsc,\xi_{d})\in\RR^{d}$.
\end{description}
We also assume that $b$ is continuous in $(x,u)$.

In integral form, \eqref{E-sde} is written as
\begin{equation}\label{E2}
X_{t} = X_{0} + \int_{0}^{t} b(X_{s},U_{s})\,\D{s}
+ \int_{0}^{t} \upsigma(X_{s})\,\D{W}_{s}\,.
\end{equation}
The second term on the right hand side of \eqref{E2} is an It\^o
stochastic integral.
We say that a process $X=\{X_{t}(\omega)\}$ is a solution of \eqref{E-sde},
if it is $\sF_{t}$-adapted, continuous in $t$, defined for all
$\omega\in\Omega$ and $t\in[0,\infty)$, and satisfies \eqref{E2} for
all $t\in[0,\infty)$ at once a.s.

We define the family of operators $\Lg^{u}:\Cc^{2}(\RR^{d})\mapsto\Cc(\RR^{d})$,
where $u\in\Act$ plays the role of a parameter, by
\begin{equation}\label{E3}
\Lg^{u} f(x) = a^{ij}(x)\,\partial_{ij} f(x)
+ b^{i}(x,u)\, \partial_{i} f(x)\,,\quad u\in\Act\,.
\end{equation}
We refer to $\Lg^{u}$ as the \emph{controlled extended generator} of
the diffusion.

Of fundamental importance in the study of functionals of $X$ is
It\^o's formula.
For $f\in\Cc^{2}(\RR^{d})$ and with $\Lg^{u}$ as defined in \eqref{E3},
it holds that
\begin{equation}\label{E4}
f(X_{t}) = f(X_{0}) + \int_{0}^{t}\Lg^{U_{s}} f(X_{s})\,\D{s}
+ M_{t}\,,\quad\text{a.s.},
\end{equation}
where
\begin{equation*}
M_{t} \df \int_{0}^{t}\bigl\langle\nabla f(X_{s}),
\upsigma(X_{s})\,\D{W}_{s}\bigr\rangle
\end{equation*}
is a local martingale.
Krylov's extension of the It\^o formula \cite[p.~122]{Krylov}
extends \eqref{E4} to functions $f$ in the local
Sobolev space $\Sobl^{2,p}(\RR^{d})$, $p\ge d$.

Recall that a control is called \emph{Markov} if
$U_{t} = v(t,X_{t})$ for a measurable map $v :\RR_{+}\times\RR^{d}\mapsto \Act$,
and it is called \emph{stationary Markov} if $v$ does not depend on
$t$, i.e., $v :\RR^{d}\mapsto \Act$.
Correspondingly, the equation 
\begin{equation}\label{E5}
X_{t} = x_{0} + \int_{0}^{t} b\bigl(X_{s},v(s,X_{s})\bigr)\,\D{s} +
\int_{0}^{t} \upsigma(X_{s})\,\D{W}_{s}
\end{equation}
is said to have a \emph{strong solution}
if given a Wiener process $(W_{t},\sF_{t})$
on a complete probability space $(\Omega,\sF,\Prob)$, there
exists a process $X$ on $(\Omega,\sF,\Prob)$, with $X_{0}=x_{0}\in\RR^{d}$,
which is continuous,
$\sF_{t}$-adapted, and satisfies \eqref{E5} for all $t$ at once, a.s.
A strong solution is called \emph{unique},
if any two such solutions $X$ and $X'$ agree
$\Prob$-a.s., when viewed as elements of $\Cc\bigl([0,\infty),\RR^{d}\bigr)$.
It is well known that under Assumptions (A1)--(A3),
for any Markov control $v$,
\eqref{E5} has a unique strong solution \cite{Gyongy-96}.

Let $\Usm$ denote the set of stationary Markov controls.
Under $v\in\Usm$, the process $X$ is strong Markov,
and we denote its transition function by $P^{t}_{v}(x,\cdot\,)$.
It also follows from the work of \cite{Bogachev-01,Stannat-99} that under
$v\in\Usm$, the transition probabilities of $X$
have densities which are locally H\"older continuous.
Thus $\Lg^{v}$ defined by
\begin{equation*}
\Lg^{v} f(x) = a^{ij}(x)\,\partial_{ij} f(x)
+ b^{i} \bigl(x,v(x)\bigr)\, \partial_{i} f(x)\,,\quad v\in\Usm\,,
\end{equation*}
for $f\in\Cc^{2}(\RR^{d})$,
is the generator of a strongly-continuous
semigroup on $\Cc_{b}(\RR^{d})$, which is strong Feller.
We let $\Prob_{x}^{v}$ denote the probability measure and
$\Exp_{x}^{v}$ the expectation operator on the canonical space of the
process under the control $v\in\Usm$, conditioned on the
process $X$ starting from $x\in\RR^{d}$ at $t=0$.

\subsection{The ergodic control problem}\label{S-ergodic}

We assume that the running
cost function $r\colon\RR^{d}\times\Act\to\RR_{+}$
is continuous and locally Lipschitz
in its first argument uniformly in $u\in\Act$.
Without loss of generality we let $\kappa_{R}$
be a Lipschitz constant of $r$ over $B_{R}$.
More specifically, we assume that
\begin{equation*}
\babs{r(x,u)-r(y,u)} \le \kappa_{R}\abs{x-y}
\qquad\forall x,y\in B_{R}\,,~\forall u\in\Act\,,
\end{equation*}
and all $R>0$.

As mentioned in \S\ref{S1},
an important class of running cost functions arising in practice
for which the ergodic control
problem is well behaved are the near-monotone cost functions.

The ergodic control problem for near-monotone cost functions is characterized
by the following theorem which we quote from \cite{book}.
Note that we choose to normalize the value function $V^{*}$ differently here,
in order to facilitate the use of weighted norms.

\begin{theorem}\label{T3.1}
There exists a unique function $V^{*}\in\Cc^{2}(\RR^{d})$
which solves the HJB equation \eqref{E-HJB},
and satisfies $\min_{\RR^{d}}\;V^{*}=1$.
Also, a control $v\in\Usm$ is optimal with respect to the criteria
\eqref{E-ergcrit} and \eqref{E-avgcrit} if and only if it satisfies
\eqref{E-v*}
a.e.\ in $\RR^{d}$.
Moreover, recalling that $\Breve{\uptau}_{R}=\uptau(B_{R}^{c})$, $R>0$, we have
\begin{equation}\label{E-strepr}
V^{*}(x) = \inf_{v\in\Ussm}\;
\Exp_{x}^{v} \left[\int_{0}^{\Breve{\uptau}_{R}}
\bigl(r\bigl(X_{t},v(X_{t})\bigr)-\varrho\bigr)\,\D{t}
+ V^{*}(X_{\uptau(B_{R}^{c})})\right] \qquad \forall x\in B_{R}^{c}\,,
\end{equation}
for all $R>0$.
\end{theorem}

Recall that control $v\in\Usm$ is called \emph{stable}
if the associated diffusion is positive recurrent.
We denote the set of such controls by $\Ussm$,
and let $\imeas_{v}$ denote the unique invariant probability
measure on $\RR^{d}$ for the diffusion under the control $v\in\Ussm$.
Recall that $v\in\Ussm$ if and only if there exists an inf-compact function
$\Lyap\in\Cc^{2}(\RR^{d})$, a bounded domain $D\subset\RR^{d}$, and
a constant $\varepsilon>0$ satisfying
\begin{equation*}
\Lg^{v}\Lyap(x) \le -\varepsilon
\qquad\forall x\in D^{c}\,.
\end{equation*}
It follows that the optimal control $v^{*}$ in Theorem~\ref{T3.1} is stable.

We make the following mild technical assumption
which is in effect throughout the paper:

\begin{assumption}\label{A3.2}
The value function $V^{*}$ is integrable with respect
to some optimal invariant probability distribution
$\imeas_{v^{*}}$.
\end{assumption}

\begin{remark}
Assumption~\ref{A3.2} is equivalent to the following \cite[Lemma~3.3.4]{book}:
there exists an optimal stationary control $v^{*}$ and an inf-compact
function $\Lyap\in\Cc^{2}(\RR^{d})$ and an
open ball $B\subset\RR^{d}$ such that
\begin{equation}\label{E-2order}
\Lg^{v^{*}} \Lyap(x) \le -V^{*}(x)\qquad \forall x\in B^{c}\,.
\end{equation}
For the rest of the paper $v^{*}\in\Ussm$
denotes some fixed control satisfying \eqref{E-v*} and \eqref{E-2order}.
\end{remark}

\begin{remark}
Assumption~\ref{A3.2} is pretty mild.
In the case that $r$ is bounded it is equivalent to the statement
that the mean hitting times to an open bounded set are integrable
with respect to some  optimal invariant probability distribution.
In the case of one dimensional
diffusions, provided $\upsigma(x)>\upsigma_{0}$
for some constant $\upsigma_{0}>0$, and
$\limsup_{\abs{x}\to\infty}\; \frac{x\,b(x)}{\upsigma^{2}(x)}<-\frac{1}{2}$,
then the mean hitting time of $0\in\RR$ is bounded above
by a second-degree polynomial in $x$
\cite[Theorem~5.6]{Locherbach-11}.
Therefore, in this case, the existence of second moments for
$\imeas_{v^{*}}$ implies Assumption~\ref{A3.2}.
\end{remark}

We need the following lemma.

\begin{lemma}\label{L3.4}
Under Assumption~\ref{A3.2},
\begin{equation*}
\Exp_{x}^{v^{*}}\bigl[V^{*}(X_{t})\bigr]
\xrightarrow[t\to\infty]{}\imeas_{v^{*}}[V^{*}] \df
\int_{\RR^{d}} V^{*}(x)\,\imeas_{v^{*}}(\D{x})\qquad \forall x\in\RR^{d}\,,
\end{equation*}
where, as defined earlier,
$\imeas_{v^{*}}$ is the invariant probability measure of the diffusion
under the control $v^{*}$.
Also there exists a constant $m_{r}$ depending on $r$ such that
\begin{equation}\label{E-deg2bound}
\sup_{t\ge0}\;\Exp^{v^{*}}_{x}\bigl[V^{*}(X_{t})\bigr]
\le m_{r}(V^{*}(x)+1)\qquad \forall x\in\RR^{d}\,.
\end{equation}
\end{lemma}

\begin{proof}
Since $r$ is nonnegative, by Dynkin's formula we have
\begin{equation}\label{E-Dynkinbound}
\Exp_{x}^{v^{*}}[V^{*}(X_{t})]\le V^{*}(x)+ \varrho\, t
\qquad\forall t\ge0\,,~\forall x\in\RR^{d}\,.
\end{equation}
Therefore, since $V^{*}$ is integrable with respect to
$\imeas_{v^{*}}$ by Assumption~\ref{A3.2},
the first result follows by \cite[Theorem~5.3~(i)]{MT-III}.
The bound in \eqref{E-deg2bound} is the continuous time analogue of
(14.5) in \cite{Meyn-09}.
Recall that a skeleton of a continuous-time Markov process is
a discrete-time Markov process with transition probability
$\widehat{P}= \int_{0}^{\infty} \alpha(\D{t})P^{t}$, where $\alpha$
is a probability measure on $(0,\infty)$.
Since the diffusion is nondegenerate, any skeleton of the process
is $\phi$-irreducible, with an irreducibility measure absolutely
continuous with respect to the Lebesgue measure
(for a definition of $\phi$-irreducibility we refer the reader to
\cite[Chapter~4]{Meyn-09}).
It is also straightforward to show that compact subsets of $\RR^{d}$ are petite.
Define the transition probability $\widetilde{P}$ by
\begin{equation*}
\widetilde{P}f(x) =
\int_{\RR^{d}}\widetilde{P}(x,\D{y})\,f(y)
\df \Exp_{x}^{v^{*}}[f(X_{t})] \Bigr|_{t=1}\,,\quad x\in\RR^{d}
\end{equation*}
for all bounded functions $f\in\Cc(\RR^{d})$, and
\begin{equation*}
g_{r}(x) \df \Exp^{v^{*}}_{x}\left[
\int_{0}^{1}r\bigl(X_{s},v^{*}(X_{s})\bigr)\,\D{s}\right]\,,
\quad x\in\RR^{d}\,.
\end{equation*}
Then \eqref{E-HJB} translates into the discrete time Poisson equation:
\begin{equation}\label{E-deg2.1}
\widetilde{P}V^{*}(x) - V^{*}(x) = \varrho-g_{r}(x)\,,\quad x\in\RR^{d}\,.
\end{equation}
It easily follows from the near-monotone hypothesis
\eqref{E-nm} that there exists a constant $\varepsilon_{0}>0$
and a ball $B_{R_{0}}\subset\RR^{d}$, $R_{0}>0$, such that
$g_{r}(x) - \varrho>\varepsilon_{0}$ for all $x\in B_{R_{0}}^{c}$.
Since, in addition, $\int_{\RR^{d}} V^{*}(x)\imeas_{v^{*}}(\D{x})<\infty$,
it follows by \cite[Theorem~14.0.1]{Meyn-09} that
there exists a constant $\widetilde{m}$ such that
\begin{equation}\label{E-deg2.2}
\sum_{n=0}^{\infty} \babs{\widetilde{P}^{n}g_{r}(x)-\varrho}
\le \widetilde{m} (V^{*}(x)+1) \qquad \forall x\in\RR^{d}\,.
\end{equation}
By \eqref{E-deg2.1}--\eqref{E-deg2.2} we obtain
\begin{align}\label{E-deg2.3}
\widetilde{P}^{n}V^{*}(x) &= V^{*}(x) - \sum_{k=0}^{n-1}
(\widetilde{P}^{k}g_{r}(x)-\varrho)
\\[5pt]\nonumber
&\le (\widetilde{m}+1) (V^{*}(x)+1)\,.
\end{align}
By \eqref{E-Dynkinbound} and \eqref{E-deg2.3},
writing the arbitrary $t\in\RR_{+}$ as $t=n+\delta$ where $n$ is
the integer part of $t$ and using the Markov property, we obtain
\begin{align*}
\Exp^{v^{*}}_{x}\bigl[V^{*}(X_{t})\bigr] &= \Exp^{v^{*}}_{x}
\Bigl[\Exp^{v^{*}}_{X_{\delta}}\bigl[V^{*}(X_{t-\delta})\bigr] \Bigr]
\nonumber\\[5pt]
&= \Exp^{v^{*}}_{x}\bigl[\widetilde{P}^{n}V^{*}(X_{\delta})\bigr]
\nonumber\\[5pt]
&\le \Exp^{v^{*}}_{x}\bigl[(\widetilde{m}+1) (V^{*}(X_{\delta})+1)\bigr]
\nonumber\\[5pt]
&\le (\widetilde{m}+1)\,( V^{*}(x)+ \varrho\, \delta+1)
\nonumber\\[5pt]
&\le (\widetilde{m}+1)\,( V^{*}(x)+ \varrho+1)
\qquad\forall t\ge0\,,~\forall x\in\RR^{d}\,,
\end{align*}
thus establishing \eqref{E-deg2bound}.
\qquad
\end{proof}

\smallskip
\begin{definition}
We let $\Cc_{V^{*}}(\RR^{d})$ denote the Banach space of functions
$f\in\Cc(\RR^{d})$ with norm
\begin{equation*}
\norm{f}_{V^{*}} \df
\sup_{x\in\RR^{d}}\;\frac{\abs{f(x)}}{V^{*}(x)}\,.
\end{equation*}
We also define
\begin{equation*}
\order_{V^{*}} \df
\bigl\{f\in\Cc_{V^{*}}(\RR^{d})\cap\Cc^{2}(\RR^{d}) : f\ge0\bigr\}\,.
\end{equation*}
\end{definition}

\subsection{The relative value iteration}\label{S-VIRVI}

The RVI and VI equations in \eqref{E-RVI} and \eqref{E-VI}
can also be written in the form
\begin{align}
\partial_{t}{\RVI}(t,x)
&= \min_{u\in\Act}\; \bigl[\Lg^{u} {\RVI}(t,x) +r(x,u)\bigr]
- {\RVI}(t,0)\,, &{\RVI}(0,x)={\RVI}_{0}(x)\,,\label{ERVI}\\[5pt]
\partial_{t} {\VI}(t,x)
&= \min_{u\in\Act}\; \bigl[\Lg^{u} {\VI}(t,x) +r(x,u)\bigr]
- \varrho\,, &{\VI}(0,x)={\RVI}_{0}(x)\,.\label{EVI}
\end{align}

\begin{definition}\label{D-minimizer}
Let $\Hv=\{\Hv_{t}\,,\ t\in\RR_{+}\}$ denote a measurable selector from the minimizer
in \eqref{EVI} corresponding to a solution ${\VI}\in\Ccl^{1,2}(\RR^{d})$.
This is also a measurable selector from the minimizer
in \eqref{ERVI}, provided  ${\VI}$ and ${\RVI}$ are related
by \eqref{E-rvivi} and \eqref{E-virvi}, and vice-versa.
Note that the Markov control associated with
$\Hv$ is computed `backward' in time (see \eqref{E-VISR}).
Hence, for each $t\ge0$ we define the (nonstationary) Markov control
\begin{equation*}
\Hv^{t}\df\bigl\{\Hv^{t}_{s}=\Hv_{t-s}\,,\; s\in[0,t]\bigr\}\,.
\end{equation*}
Also, we adopt the simplifying notation
\begin{equation*}
\Or(x,u) \df r(x,u) - \varrho\,.
\end{equation*}
\end{definition}

In most of the statements of intermediary results
the initial data ${\RVI}_{0}$ is assumed without loss of generality
to be nonnegative.
We start with a theorem that proves the existence of a
solution to \eqref{EVI} that admits the stochastic
representation in \eqref{E-VISR}.
This does not require Assumption~\ref{A3.2}.

First we need the following definition.

\begin{definition}
We define $\RR^{d}_{T}\df(0,T)\times\RR^{d}$,
and let $\overline{\RR^{d}_{T}}$ denote its closure.
We also let $\Cc^{T}_{V^{*}}(\RR^{d})$ denote the Banach space of functions
in $\Cc(\overline{\RR^{d}_{T}})$ with norm
\begin{equation*}
\norm{f}_{V^{*},T} \df \sup_{(t,x)\in \overline{\RR^{d}_{T}}}\;
\frac{\abs{f(t,x)}}{V^{*}(x)}\,.
\end{equation*}
\end{definition}

\begin{theorem}\label{T-canonical}
Provided ${\RVI}_{0}\in\order_{V^{*}}$, then
\begin{subequations}
\begin{align}
{\VI}(t,x) &= \inf_{U\in\Uadm}\;\Exp^{U}_{x}
\left[\int_{0}^{t}\Or(X_{s},U_{s})\,\D{s}
+{\RVI}_{0}(X_{t})\right]\,,\label{E-SRa}
\intertext{is the minimal solution of \eqref{EVI} in
$\Ccl^{1,2}\bigl((0,\infty)\times \RR^{d}\bigr)
\cap\Cc\bigl([0,\infty)\times\RR^{d}\bigr)$ which is bounded
below on $\RR^{d}_{T}$, for any $T>0$.
With $\Hv^{t}$ as defined in
Definition~\ref{D-minimizer}, it admits the representation}
{\VI}(t,x) &= \Exp^{\Hv^{t}}_{x}
\left[\int_{0}^{t}\Or\bigl(X_{s},\Hv^{t}_{s}(X_{s})\bigr)\,\D{s}
+{\RVI}_{0}(X_{t})\right]\,,\label{E-SRb}
\end{align}
\end{subequations}
and it holds that
\begin{equation}\label{E-SRc}
\Exp_{x}^{\Hv^{t}}\bigl[{\VI}(t-\uptau_{R},X_{\uptau_{R}})\,
\Ind\{\uptau_{R}<t\}\bigr]\xrightarrow[R\to\infty]{}0
\end{equation}
for all $(t,x)\in\RR_{+}\times\RR^{d}$.
Moreover ${\VI}(t,\cdot\,)\ge-\varrho\,t$ and satisfies the estimate
\begin{equation}\label{E-Vest}
\norm{{\VI}}_{V^{*},T} \le (1+\varrho\,T)\;
\max\;\bigl(1,\norm{{\RVI}_{0}}_{V^{*}}\bigr)
\qquad \forall T>0\,.
\end{equation}
\end{theorem}

\begin{proof}
Let $r^{n}$ and ${\RVI}^{n}_{0}$, for $n\in\NN$,
be smooth truncations of $r$ and ${\RVI}_{0}$, respectively,
satisfying $\norm{r^{n}}_\infty\le n$ and $\norm{{\RVI}^{n}_{0}}_\infty\le n$
and such that
$r^{n}\uparrow r$ and ${\RVI}^{n}_{0}\uparrow{\RVI}_{0}$
as $n\to\infty$.
Let $\varrho_{n}$ denote the optimal ergodic
cost corresponding to $r^{n}$.
The boundary value problem
\begin{equation}\label{E-Dir1}
\begin{gathered}
\partial_{t} {\widehat{\RVI}}^{R}_{n}(t,x)
= \min_{u\in\Act}\; \left[\Lg^{u} {\widehat{\RVI}}^{R}_{n}(t,x) +r(x,u)\right]
\qquad \text{in}\ (0,T)\times B_{R}\\[5pt]
{\widehat{\RVI}}^{R}_{n}(0,x)={\RVI}^{n}_{0}(x)
\quad\forall x\in\overline{B}_{R}\,,\qquad
{\widehat{\RVI}}^{R}_{n}(t,\cdot\,)\vert_{\partial B_{R}}= {\RVI}^{n}_{0}
\quad\forall t\in[0,T]\,,
\end{gathered}
\end{equation}
has a unique nonnegative solution in
$\Cc^{1,2}\bigl((0,T)\times B_{R}\bigr)
\cap\Cc\bigl([0,T]\times\overline{B}_{R}\bigr)$
for all $T>0$ and $R>0$.
This solution has the stochastic representation
\begin{equation}\label{E-Dir2}
{\widehat{\RVI}}^{R}_{n}(t,x)
= \inf_{U\in\Uadm}\;\Exp_{x}^{U}
\left[\int_{0}^{\uptau_{R}\wedge{t}}r^{n}(X_{s},U_{s})\,\D{s}
+{\RVI}^{n}_{0}(t-\uptau_{R}\wedge{t},X_{\uptau_{R}\wedge{t}})\right]\,.
\end{equation}
where, as defined in \S\ref{S-not},
$\uptau_{R}$ denotes the first exit time from the ball $B_{R}$.
By \eqref{E-Dir2} we obtain
\begin{align*}
{\widehat{\RVI}}^{R}_{n}(t,x)
&\le \Exp_{x}^{v^{*}}
\left[\int_{0}^{\uptau_{R}\wedge{t}}r^{n}\bigl(X_{s},v^{*}(X_{s})\bigr)\,\D{s}
+{\RVI}^{n}_{0}(t-\uptau_{R}\wedge{t},X_{\uptau_{R}\wedge{t}})\right]
\\[5pt]
&\le \max\;\bigl(1,\norm{{\RVI}_{0}}_{V^{*}}\bigr)
\Exp_{x}^{v^{*}}
\left[\int_{0}^{\uptau_{R}\wedge{t}}r\bigl(X_{s},v^{*}(X_{s})\bigr)\,\D{s}
+V^{*}(t-\uptau_{R}\wedge{t},X_{\uptau_{R}\wedge{t}})\right]
\\[5pt]
&\le  \max\;\bigl(1,\norm{{\RVI}_{0}}_{V^{*}}\bigr)
\bigl(V^{*}(x)+\varrho\,t\bigr)\,.
\end{align*}
Therefore by the interior estimates of solutions
of \eqref{E-Dir1} (see \cite[Theorem~5.1]{Lady})
the derivatives $\bigl\{D^{\alpha}\partial_{t}^{\ell}{\widehat{\RVI}}^{R}_{n}:
\abs{\alpha}+2\ell \le 2\,,~R>0\,,n\in\NN\bigl\}$ are
locally H\"older equicontinuous in $\RR^{d}_{T}$.
Thus passing to the limit as $R\to\infty$ along a subsequence
we obtain a nonnegative function
${\widehat{\RVI}}_{n}\in\Ccl^{1,2}\bigl(\RR^{d}_{T}\bigl)
\cap\,\Cc\bigl(\overline{\RR^{d}_{T}}\bigr)$,
for all $T>0$, which satisfies
\begin{equation}\label{E-Dir3}
\begin{aligned}
\partial_{t} {\widehat{\RVI}}_{n}(t,x)
&= \min_{u\in\Act}\; \left[\Lg^{u} {\widehat{\RVI}}_{n}(t,x) + r^{n}(x,u)\right]
\qquad \text{in}\ (0,\infty)\times\RR^{d} \\[5pt]
{\widehat{\RVI}}_{n}(0,x)
&={\RVI}^{n}_{0}(x)\qquad\forall x\in\RR^{d}\,.
\end{aligned}
\end{equation}

By using Dynkin's formula on the cylinder $[0,t]\times B_{R}$, we obtain
from \eqref{E-Dir3} that
\begin{equation}\label{E-SR1}
{\widehat{\RVI}}_{n}(t,x) =\inf_{U\in\Uadm}\;\Exp_{x}^{U}
\left[\int_{0}^{\uptau_{R}\wedge{t}}r^{n}(X_{s},U_{s})\,\D{s}
+{\widehat{\RVI}}_{n}(t-\uptau_{R}\wedge{t},X_{\uptau_{R}\wedge{t}})\right]\,,
\end{equation}
It follows by \eqref{E-Dir2}  that
$\norm{{\widehat{\RVI}}_{n}(t,\,\cdot)}_{\infty}\le n(t+1)$
for all $n\in\NN$ and $t\ge0$.
By \eqref{E-SR1} we have the inequality
\begin{align}\label{E-SR2}
{\widehat{\RVI}}_{n}(t,x) &\le \Exp_{x}^{U}
\left[\int_{0}^{\uptau_{R}\wedge{t}}r^{n}(X_{s},U_{s})\,\D{s}
+{\widehat{\RVI}}_{n}(t-\uptau_{R}\wedge{t},X_{\uptau_{R}\wedge{t}})\right]
\\[5pt]\notag
&\le \Exp_{x}^{U}
\left[\int_{0}^{\uptau_{R}\wedge{t}}r^{n}(X_{s},U_{s})\,\D{s}
+ {\RVI}^{n}_{0}(X_{t})\,\Ind\{\uptau_{R}>t\}\right]
+ n \Prob_{x}^{U}(\uptau_{R}\le t)
\end{align}
for all $U\in\Uadm$.
Taking limits as $R\to\infty$ in \eqref{E-SR2},
using dominated convergence,
we obtain
\begin{equation}\label{E-SR3}
{\widehat{\RVI}}_{n}(t,x) \le
\Exp_{x}^{U}
\left[\int_{0}^{t}r^{n}(X_{s},U_{s})\,\D{s}
+ {\RVI}^{n}_{0}(X_{t})\right]
\qquad U\in\Uadm\,.
\end{equation}

Note that
\begin{equation}\label{E-SR4}
0\le{\widehat{\RVI}}_{n}(t,x)\le \limsup_{R\to\infty}\;
{\widehat{\RVI}}^{R}_{n}(t,x)
\le \max\;\bigl(1,\norm{{\RVI}_{0}}_{V^{*}}\bigr)\bigl(V^{*}(x)+\varrho\,t\bigr)\,.
\end{equation}
Hence, as mentioned earlier,
the derivatives $\bigl\{D^{\alpha}\partial_{t}^{\ell}{\widehat{\RVI}}_{n}:
\abs{\alpha}+2\ell \le 2\,,~n\in\NN\bigl\}$
are locally H\"older equicontinuous in $(0,\infty)\times\RR^{d}$.
Also as shown in \cite[p.~119]{book} we have $\varrho_{n}\to\varrho$ as $n\to\infty$.
Let $\{k_{n}\}_{n\in\NN}\subset\NN$ be an arbitrary sequence.
Then there exists some subsequence $\{k'_{n}\}\subset\{k_{n}\}$
such that
${\widehat{\RVI}}_{k'_{n}}\to {\widehat{\RVI}}\in\Ccl^{1,2}\bigl(\RR^{d}_{T}\bigl)
\cap\,\Cc\bigl(\overline{\RR^{d}_{T}}\bigr)$,
for all $T>0$, and ${\widehat{\RVI}}$ satisfies
\begin{equation}\label{E-Dir4}
\begin{aligned}
\partial_{t} {\widehat{\RVI}}(t,x)
&= \min_{u\in\Act}\; \left[\Lg^{u} {\widehat{\RVI}}(t,x) + r(x,u)\right]
\qquad \text{in}\ (0,\infty)\times\RR^{d} \\[5pt]
{\widehat{\RVI}}(0,x)
&={\RVI}_{0}(x)\qquad\forall x\in\RR^{d}\,.
\end{aligned}
\end{equation}
Let $\Hv^{t}$ denote a stationary Markov control associated with
the minimizer in \eqref{E-Dir4} as in Definition~\ref{D-minimizer}.
By using Dynkin's formula on the cylinder $[0,t]\times B_{R}$, we obtain
from \eqref{E-Dir4}
\begin{subequations}
\begin{align}
{\widehat{\RVI}}(t,x) &=\inf_{U\in\Uadm}\;\Exp_{x}^{U}
\left[\int_{0}^{\uptau_{R}\wedge{t}}r(X_{s},U_{s})\,\D{s}
+{\widehat{\RVI}}(t-\uptau_{R}\wedge{t},X_{\uptau_{R}\wedge{t}})\right]\,,
\label{E-SR5a}
\\[5pt]
{\widehat{\RVI}}(t,x) &= \Exp^{\Hv^{t}}_{x}
\left[\int_{0}^{\uptau_{R}\wedge{t}}
r\bigl(X_{s},\Hv^{t}_{s}(X_{s})\bigr)\,\D{s}
+{\widehat{\RVI}}(t-\uptau_{R}\wedge{t},X_{\uptau_{R}\wedge{t}})\right]\,.
\label{E-SR5b}
\end{align}
\end{subequations}
Since ${\widehat{\RVI}}(t,\cdot\,)$ is nonnegative,
letting $R\to\infty$ in \eqref{E-SR5b}, by Fatou's Lemma we obtain
\begin{align}\label{E-SR6}
{\widehat{\RVI}}(t,x) &\ge \Exp^{\Hv^{t}}_{x}
\left[\int_{0}^{t}r\bigl(X_{s},\Hv^{t}_{s}(X_{s})\bigr)\,\D{s}
+{\RVI}_{0}(X_{t})\right]
\\[5pt]\nonumber
&\ge \inf_{U\in\Uadm}\;\Exp^{U}_{x}
\left[\int_{0}^{t}r(X_{s},U_{s})\,\D{s}+{\RVI}_{0}(X_{t})\right]\,.
\end{align}

Taking limits as $n\to\infty$ in \eqref{E-SR3}, using monotone convergence
for the first term on the right hand side, we obtain
\begin{equation}\label{E-SR7}
{\widehat{\RVI}}(t,x)
\le \Exp_{x}^{U} \left[\int_{0}^{t}r(X_{s},U_{s})\,\D{s}
+{\RVI}_{0}(X_{t})\right]\qquad \forall U\in\Uadm\,.
\end{equation}
By \eqref{E-SR6}--\eqref{E-SR7}
we have
\begin{subequations}
\begin{align}
{\widehat{\RVI}}(t,x)
&= \inf_{U\in\Uadm}\;\Exp^{U}_{x}
\left[\int_{0}^{t}r(X_{s},U_{s})\,\D{s}+{\RVI}_{0}(X_{t})\right]\label{E-SR8a}
\\[5pt]
{\widehat{\RVI}}(t,x)&= \Exp^{\Hv^{t}}_{x}
\left[\int_{0}^{t}r\bigl(X_{s},\Hv^{t}_{s}(X_{s})\bigr)\,\D{s}
+{\RVI}_{0}(X_{t})\right]\label{E-SR8b}
\,.
\end{align}
\end{subequations}

Let ${\VI}(t,x) \df {\widehat{\RVI}}(t,x) - \varrho\,t$.
Then ${\VI}$ solves \eqref{EVI} and
\eqref{E-SRa}--\eqref{E-SRb} follow by \eqref{E-SR8a}--\eqref{E-SR8b}.
It is also clear that ${\VI}(t,x)\ge -\varrho\,t$, which together with \eqref{E-SR4}
implies \eqref{E-Vest}.

By \eqref{E-SR5a} we have
\begin{align}\label{E-SR9}
{\widehat{\RVI}}(t,x) &= \Exp_{x}^{\Hv^{t}}
\left[\int_{0}^{\uptau_{R}\wedge{t}}r\bigl(X_{s},\Hv^{t}(X_{s})\bigr)
\,\D{s}
+{\RVI}_{0}(t,X_{t})\,\Ind\{\uptau_{R}\ge t\}\right]
\\[5pt]\nonumber
&\mspace{250mu}
+ \Exp_{x}^{\Hv^{t}}\bigl[{\widehat{\RVI}}(t-\uptau_{R},X_{\uptau_{R}})\,
\Ind\{\uptau_{R}<t\}\bigr]\,.
\end{align}
The first term on the right hand side of \eqref{E-SR9} tends
to the right hand side of \eqref{E-SR8b} by monotone convergence as $R\uparrow\infty$.
Therefore \eqref{E-SRc} holds.

Suppose ${\widetilde{\RVI}}$ is a solution of
\eqref{E-Dir4} in
$\Ccl^{1,2}\bigl(\RR^{d}_{T}\bigl)
\cap\,\Cc\bigl(\overline{\RR^{d}_{T}}\bigr)$, for some $T>0$, which is bounded below,
and $\Tilde{v}^{t}$ is an associated stationary Markov
control from the minimizer of \eqref{E-Dir4}.
Applying Dynkin's formula on the cylinder
$[0,t]\times B_{R}$ and letting $R\to\infty$ using Fatou's lemma, we obtain
\begin{align*}
{\widetilde{\RVI}}(t,x)&\ge \Exp^{\Tilde{v}^{t}}_{x}
\left[\int_{0}^{t}r\bigl(X_{s},\Hv^{t}_{s}(X_{s})\bigr)\,\D{s}
+{\RVI}_{0}(X_{t})\right]
\\[5pt]
&\ge\inf_{U\in\Uadm}\;\Exp^{U}_{x}
\left[\int_{0}^{t}r(X_{s},U_{s})\,\D{s}+{\RVI}_{0}(X_{t})\right]
\\[5pt]
&\ge {\widehat{\RVI}}(t,x)\,.
\end{align*}
Therefore ${\VI}(t,x)$
is the minimal solution of \eqref{EVI}
in $\Ccl^{1,2}\bigl((0,\infty)\times \RR^{d}\bigr)
\cap\Cc\bigl([0,\infty)\times\RR^{d}\bigr)$ which is bounded
below on $\RR^{d}_{T}$, for each $T>0$.
\qquad
\end{proof}

In the interest of economy of language we refer to
the solution in \eqref{E-SRa} as canonical.
This is detailed in the following definition.

\begin{definition}\label{D-canonical}
Given an initial condition ${\RVI}_{0}\in\order_{V^{*}}$  we define
the \emph{canonical solution} to the VI in \eqref{EVI} as
the solution which was constructed in the proof of Theorem~\ref{T-canonical}
and was shown to admit the stochastic representation in \eqref{E-SRa}.
In other words, this is the minimal solution of \eqref{EVI} in
$\Ccl^{1,2}\bigl((0,\infty)\times \RR^{d}\bigr)
\cap\,\Cc\bigl([0,\infty)\times\RR^{d}\bigr)$ which is bounded
below on $\RR^{d}_{T}$, for any $T>0$.
The canonical solution to the VI
well defines the \emph{canonical solution} to the RVI in \eqref{ERVI}
via \eqref{E-virvi}.
\end{definition}

For the rest of the paper a solution to the RVI or VI
is always meant to be a canonical solution.
In summary, these are characterized by:
\begin{align}\label{E-SR}
{\RVI}(t,x) + \int_{0}^{t} {\RVI}(s,0)\,\D{s} &= \inf_{U\in\Uadm}\;\Exp^{U}_{x}
\left[\int_{0}^{t}r(X_{s},U_{s})\,\D{s} + {\RVI}_{0}(X_{t})\right]
\\[5pt]\nonumber
&= \int_{0}^{t}\Exp^{\Hv^{t}}_{x}
\bigl[r\bigl(X_{s},\Hv^{t}_{s}(X_{s})\bigr)\bigr]\,\D{s}
+ \Exp^{\Hv^{t}}_{x}\bigl[{\RVI}_{0}(X_{t})\bigr]\,.
\end{align}
Similarly
\begin{align*}
{\VI}(t,x) &= \inf_{U\in\Uadm}\;\Exp^{U}_{x}
\left[\int_{0}^{t}\Or(X_{s},U_{s})\,\D{s}
+{\RVI}_{0}(X_{t})\right]\nonumber\\[5pt]
&= \int_{0}^{t}\Exp^{\Hv^{t}}_{x}
\bigl[\Or\bigl(X_{s},\Hv^{t}_{s}(X_{s})\bigr)\bigr]\,\D{s}
+ \Exp^{\Hv^{t}}_{x}\bigl[{\RVI}_{0}(X_{t})\bigr]\,.
\end{align*}

The next lemma provides an important estimate for the canonical solutions of the
the VI.

\begin{lemma}
Provided ${\RVI}_{0}\in\Cc_{V^{*}}(\RR^{d})\cap\Cc^{2}(\RR^{d})$,
then the canonical solution ${\VI}\in\Ccl^{1,2}\bigl((0,\infty)\times\RR^{d}\bigr)
\cap\Cc\bigl([0,\infty)\times\RR^{d}\bigr)$ of \eqref{EVI} satisfies
the bound
\begin{equation}\label{E-bound}
\Exp_{x}^{\Hv^{t}}\bigl[{\RVI}_{0}(X_{t})-V^{*}(X_{t})\bigr]
\le {\VI}(t,x) - V^{*}(x)
\le \Exp_{x}^{v^{*}}\bigl[{\RVI}_{0}(X_{t})-V^{*}(X_{t})\bigr]
\end{equation}
for all $(t,x)\in\RR_{+}\times\RR^{d}$.
\end{lemma}

\begin{proof}
By \eqref{E-HJB} and \eqref{EVI} we obtain
\begin{subequations}
\begin{align*}
-\partial_{t}(V^{*}-{\VI}) + \Lg^{v^{*}}(V^{*}-{\VI})&\le 0
\intertext{and}
-\partial_{t}(V^{*}-{\VI}) + \Lg^{\Hv^{t}}(V^{*}-{\VI})&\ge 0
\end{align*}
\end{subequations}
from which, by an application of It\^o's formula to
$V^{*}(X_{s}) - {\VI}(t-s, X_{s})$, $s\in[0,t]$, it follows that
\begin{subequations}
\begin{align*}
\Exp_{x}^{v^{*}}\bigl[V^{*}(X_{t}) -{\RVI}_{0}(X_{t})\bigr]
&\le V^{*}(x) - {\VI}(t,x)
\intertext{and}
\Exp_{x}^{\Hv^{t}}\bigl[V^{*}(X_{t}) -{\RVI}_{0}(X_{t})\bigr]
&\ge V^{*}(x) - {\VI}(t,x)\,,
\end{align*}
\end{subequations}
respectively, and the estimate follows.
\qquad
\end{proof}

Concerning the uniqueness of the canonical solution in a larger class of
functions, this depends on the growth of $V^{*}$ and the coefficients
of the SDE in \eqref{E-sde}.
Various such uniqueness results can be given based on different
hypotheses on the growth of the data.
The following result assumes that $V^{*}$ has polynomial growth,
which is the case in many applications.

\begin{theorem}
Let ${\RVI}_{0}\in\order_{V^{*}}$ and suppose that for some
constants $c_{1}$, $c_{2}$ and $m>0$, $V^{*}(x)\le c_{1} + c_{2}\abs{x}^{m}$.
Then any solution ${\VI}'\in\Ccl^{1,2}\bigl(\RR^{d}_{T}\bigl)\cap\,
\Cc\bigl(\overline{\RR^{d}_{T}}\bigr)$ of \eqref{EVI}, for some $T>0$,
 which is bounded below in $\overline{\RR^{d}_{T}}$ and satisfies
$\norm{{\VI}'}_{V^{*},T}<\infty$ agrees with the canonical solution
${\VI}$ on $\overline{\RR^{d}_{T}}$.
\end{theorem}

\begin{proof}
Let ${\VI}'$ be a solution satisfying the hypothesis in the theorem,
and let ${\VI}$ be the canonical solution of \eqref{EVI} and $\Hv^{t}$ the associated
Markov control as in Definition~\ref{D-minimizer}.
Let ${\VI}_{\varepsilon}$, for $\varepsilon>0$,
denote the canonical solution of \eqref{EVI} with initial
data ${\RVI}_{0} + \varepsilon V^{*}$ and $\Hv_{\varepsilon}$
the associated minimizer.
By Theorem~\ref{T-canonical} for each $\varepsilon>0$ we obtain
\begin{align*}
{\VI}_{\varepsilon}(t,x)
&= \inf_{U\in\Uadm}\;\Exp^{U}_{x}
\left[\int_{0}^{t}\Or\bigl(X_{s},U_{s}\bigr)\,\D{s}
+{\RVI}_{0}(X_{t}) + \varepsilon V^{*}(X_{t})\right]
\\[5pt]
&\ge -\varrho\,t + \varepsilon
\inf_{U\in\Uadm}\;\Exp^{U}_{x}
\left[\int_{0}^{t}\Or\bigl(X_{s},U_{s}\bigr)\,\D{s} + V^{*}(X_{t})\right]
\\[5pt]
&\ge \varepsilon V^{*}(x) -\varrho\,t\,.
\end{align*}
Therefore by \eqref{E-SRc} for each $\varepsilon>0$, we have
\begin{equation*}
\Exp_{x}^{\Hv_{\varepsilon}^{t}}\bigl[V^{*}(t-\uptau_{R},X_{\uptau_{R}})\,
\Ind\{\uptau_{R}<t\}\bigr]\xrightarrow[R\to\infty]{}0\qquad \forall
(t,x)\in\RR^{d}_{T}\,,
\end{equation*}
which in turn implies, since $\norm{{\VI}'}_{V^{*},T}<\infty$, that
\begin{equation}\label{E-unq0a}
\Exp_{x}^{\Hv_{\varepsilon}^{t}}\bigl[{\VI}'(t-\uptau_{R},X_{\uptau_{R}})\,
\Ind\{\uptau_{R}<t\}\bigr]\xrightarrow[R\to\infty]{}0\qquad \forall
(t,x)\in\RR^{d}_{T}\,.
\end{equation}
Since $-\partial_{t} {\VI}' +\Lg^{\Hv_{\varepsilon}^{t}}{\VI}'
+\Or(x,\Hv_{\varepsilon,t}(x)\bigr)\ge0$,
we have that for all $(t,x)\in\RR^{d}_{T}$,
\begin{equation}\label{E-unq0b}
{\VI}'(t,x)\le \Exp_{x}^{\Hv_{\varepsilon}^{t}}
\left[\int_{0}^{\uptau_{R}\wedge{t}}
\Or\bigl(X_{s},\Hv_{\varepsilon,s}^{t}(X_{s})\bigr)\,\D{s}
+{\VI}'(t-\uptau_{R}\wedge{t},X_{\uptau_{R}\wedge{t}})\right]\,,
\end{equation}
and taking limits as $R\to\infty$ in \eqref{E-unq0b}, using \eqref{E-unq0a},
it follows that ${\VI}'\le{\VI}_{\varepsilon}$ on $\overline{\RR^{d}_{T}}$.

The polynomial growth of $V^{*}$ implies that there exists a constant
$m(x,T)$ such that $\Exp_{x}^{U}[V^{*}(X_{t})]\le m(x,T)$ for all
$(t,x)\in\RR^{d}_{T}$
and $U\in\Uadm$ \cite[Theorem~2.2.2]{book}.
Therefore, since
\begin{align}\label{E-unq0c}
{\VI}_{\varepsilon}(t,x)&\le
\Exp_{x}^{\Hv^{t}}
\left[\int_{0}^{t}
\Or\bigl(X_{s},\Hv_{s}^{t}(X_{s})\bigr)\,\D{s}
+{\RVI}_{0}(X_{t}) + \varepsilon V^{*}(X_{t})\right]
\\[5pt]\nonumber
&\le {\VI}(t,x) +\varepsilon m(x,T)\qquad \forall (t,x)\in\RR^{d}_{T}\,,
\end{align}
and ${\VI}_{\varepsilon}\ge{\VI}$,
it follows by \eqref{E-unq0c} that ${\VI}_{\varepsilon}\to{\VI}$
on $\overline{\RR^{d}_{T}}$ as $\varepsilon\downarrow0$.
Thus ${\VI}'\le{\VI}$ on $\overline{\RR^{d}_{T}}$, and
by the minimality of ${\VI}$ we must have equality.
\qquad
\end{proof}

We can also obtain a uniqueness result on a larger class of functions
that does not require $V^{*}$ to have polynomial growth, but assumes
that the diffusion matrix is bounded in $\RR^{d}$.
This is given in Theorem~\ref{T-unq} below, whose proof
uses the technique in \cite{Eidelman-00}.
 
We define the following class of functions:
\begin{equation*}
\mathfrak{G} \df \bigl\{f\in\Cc^{2}(\RR^{d}) : \lim_{\abs{x}\to\infty}\;
f(x)\,\E^{-k\abs{x}^{2}}= 0\,,~ \text{for some}~ k>0\bigr\}\,.
\end{equation*}

\begin{theorem}\label{T-unq}
Suppose $V^{*}\in\mathfrak{G}$ and that
$\norm{\upsigma}$ is bounded in $\RR^{d}$.
Then, provided ${\RVI}_{0}\in\order_{V^{*}}$,
there exists a unique solution $\psi$ to \eqref{EVI}
such that $\max_{t\in[0,T]}\;\psi(t,\cdot\,)\in\mathfrak{G}$
for each $T>0$.
\end{theorem}

\begin{proof}
Let $\widehat{\RVI}\in\Ccl^{1,2}\bigl((0,\infty)\times\RR^{d}\bigl)
\cap\,\Cc\bigl([0,\infty)\times\RR^{d}\bigr)$
be the minimal nonnegative solution of
\begin{align}\label{E-unq1}
\partial_{t} \widehat{\RVI}(t,x)
&= \min_{u\in\Act}\; \left[\Lg^{u} \widehat{\RVI}(t,x) +r(x,u)\right]
\qquad \text{in}\ (0,\infty)\times\RR^{d}\,,
\\[5pt]\nonumber
\widehat{\RVI}(0,x)
&={\RVI}_{0}(x)\qquad\forall x\in\RR^{d}\,,
\end{align}
and let $\{\Hv_{t}\,,\ t\in\RR_{+}\}$ denote a measurable selector
from the minimizer in \eqref{E-unq1}.
Suppose that $\widetilde{\RVI}\in\Ccl^{1,2}\bigl((0,\infty)\times\RR^{d}\bigl)
\cap\,\Cc\bigl([0,\infty)\times\RR^{d}\bigr)$ is any solution of \eqref{E-unq1}
satisfying the hypothesis of the theorem,
and let $\{\Tilde{v}_{t}\,,\ t\in\RR_{+}\}$ denote a measurable selector
from the corresponding minimizer.
Then
$f \df \widetilde{\RVI}-\widehat{\RVI}$ satisfies, for any $T>0$,
\begin{equation}\label{E-unq2}
\partial_{t} f-\Lg^{\Hv^{T}} f\le0\quad\text{and}\quad
\partial_{t} f-\Lg^{\Tilde{v}^{T}} f\ge0
\quad \text{in}\ (0,T]\times\RR^{d}\,,
\end{equation}
and $f(0,x)=0$ for all $x\in\RR^{d}$.
By \eqref{E-Vest}, the hypothesis that $V^{*}\in\mathfrak{G}$,
and the hypothesis on the growth of $f$, it follows that for
some $k=k(T)>0$ large enough
\begin{equation}\label{E-unq3}
\lim_{\abs{x}\to\infty}\;\max_{t\in[0,T]}\;
\abs{f(t,x)}\,\E^{-k\abs{x}^{2}}=0\,.
\end{equation}

It is straightforward to verify by direct computation
using the bounds on the coefficients of the SDE that there
exists $\gamma=\gamma(k)>1$  such that
$g(t,x)\df \E^{(1+\gamma t)(1 + k\abs{x}^{2})}$ is a supersolution of
\begin{equation}\label{E-unq4}
\partial_{t} g-\Lg^{\Hv^{T_{0}}} g\ge0
\quad \text{in}\ (0,T_{0}]\times\RR^{d}\,,\quad\text{with}\ T_{0}\equiv \gamma^{-1}\,.
\end{equation}
By \eqref{E-unq3}, for any $\varepsilon>0$ we can select
$R>0$ large enough such that $\abs{f(t,x)}\le \varepsilon g(t,x)$
for all $(t,x)\in[0,\gamma^{-1}]\times\partial{B}_{R}$.
Using \eqref{E-unq2}, \eqref{E-unq4} and Dynkin's formula on the
strip $[0,\gamma^{-1}]\times\overline{B}_{R}$ it follows
that $\abs{f(t,x)}\le \varepsilon g(t,x)$ for all
$(t,x)\in[0,\gamma^{-1}]\times\overline{B}_{R}$.
Since $\varepsilon>0$ was arbitrary this implies $f\equiv0$,
or equivalently that $f =\widehat{\RVI}$
on $[0,\gamma^{-1}]\times\RR^{d}$.

Since, by \eqref{E-Vest}, $\widehat{\RVI}(\gamma^{-1},\cdot\,)\in\order_{V^{*}}$,
we can repeat the argument to show that
$f =\widehat{\RVI}$
on $[\gamma^{-1},2\gamma^{-1}]\times\RR^{d}$, and that the same holds by
induction on $[n\gamma^{-1},(n+1)\gamma^{-1}]\times\RR^{d}$, $n=2,3,\dotsc$,
until we cover the interval $[0,T]$.
This shows that $f =\widehat{\RVI}$
on $\overline{\RR^{d}_{T}}$, and since $T>0$ was arbitrary
the same holds on $[0,\infty)\times\RR^{d}$.
\qquad
\end{proof}

We do not enforce any of the assumptions of Theorem~\ref{T-unq} in
the rest of the paper.
Rather our analysis is based on the canonical solution to the VI and RVI
which is well defined (see Definition~\ref{D-canonical}).

\subsection{A region of attraction for the VI algorithm}\label{S-attract}

In this section we describe a region of attraction for the VI algorithm.
This is an subset of $\Cc^{2}(\RR^{d})$ which is invariant under
the semiflow defined by \eqref{EVI} and all its points are convergent,
i.e., converge to a solution of \eqref{E-HJB}.

\begin{definition}\label{D-flow}
We let $\overline\Phi_{t}[{\RVI}_{0}]:\Cc^{2}(\RR^{d})\to\Cc^{2}(\RR^{d})$,
$t\in[0,\infty)$, denote the canonical solution (semiflow) of
the VI in \eqref{EVI} starting from ${\RVI}_{0}$,
and $\Phi_{t}[{\RVI}_{0}]$ denote the corresponding
canonical solution (semiflow) of the RVI in \eqref{ERVI}.
Let $\Equil$ denote the set of solutions of the HJB in \eqref{E-HJB}, i.e.,
\begin{equation*}
\Equil\df\{V^{*} + c: c\in\RR\}\,.
\end{equation*}
Also for $c\in\RR$ we define the set $\cG_{c}\subset\Cc^{2}(\RR^{d})$ by
\begin{equation*}
\cG_{c} \df\bigl\{h\in\Cc^{2}(\RR^{d}) :
h-V^{*}\ge c\,,~ \norm{h}_{V^{*}}<\infty\bigr\}\,.
\end{equation*}
\end{definition}

We claim that for each $c\in\RR$, $\cG_{c}$ is invariant under the flow
$\overline\Phi_{t}$.
Indeed by \eqref{E-deg2bound} and \eqref{E-bound}, if ${\RVI}_{0}\in\cG_{c}$,
then we have that
\begin{align*}
c\le \overline\Phi_{t}[{\RVI}_{0}](x) - V^{*}(x)
&\le \Exp_{x}^{v^{*}}\bigl[{\RVI}_{0}(X_{t})-V^{*}(X_{t})\bigr]
\\[5pt]\nonumber
&\le
m_{r} \norm{{\RVI}_{0}-V^{*}}_{V^{*}}
(V^{*}(x)+1)\qquad \forall (t,x)\in\RR_{+}\times\RR^{d}\,.
\end{align*}
Since translating ${\RVI}_{0}$ by a constant simply translates the orbit
$\overline\Phi_{t}[{\RVI}_{0}]$, without loss of generality we let $c=0$, and we
show that all the points of $\cG_{0}$ are convergent.

\smallskip
\begin{theorem}\label{T3.10}
Under Assumption~\ref{A3.2},
for each ${\RVI}_{0}\in\cG_{0}$ the orbit
$\overline\Phi_{t}[{\RVI}_{0}]$, and therefore
also $\Phi_{t}[{\RVI}_{0}]$, converges as $t\to\infty$ to a point in
$\Equil\cap\cG_{0}$.
\end{theorem}

\begin{proof}
Since, as we showed in the paragraph preceding the theorem,
$\overline\Phi_{t}[{\RVI}_{0}]\in\cG_{0}$ for all $t\ge0$, by 
\eqref{E-SRa} we have
\begin{equation}\label{ET3.10a}
\overline\Phi_{t}[{\RVI}_{0}](x) \le  \Exp^{v^{*}}_{x}
\left[\int_{0}^{t-\tau}\Or\bigl(X_{s},v^{*}(X_{s})\bigr)\,\D{s}
+\overline\Phi_{\tau}[{\RVI}_{0}](X_{t-\tau})\right]\qquad \forall \tau\in[0,t]\,,
\end{equation}
Since $\overline\Phi_{t}[{\RVI}_{0}](x)- V^{*}(x)\ge0$,
and $\int_{\RR^{d}} \overline\Phi_{t}[{\RVI}_{0}](x)
\,\imeas_{v^{*}}(\D{x})$ is finite by Assumption~\ref{A3.2}, it follows by integrating
\eqref{ET3.10a} with respect to $\imeas_{v^{*}}$ that the map
\begin{equation}\label{E-attr1}
t\mapsto \int_{\RR^{d}} \overline\Phi_{t}[{\RVI}_{0}](x)
\,\imeas_{v^{*}}(\D{x})
\end{equation}
is nonincreasing and bounded below.
Hence it must be constant on the $\omega$-limit set of ${\RVI}_{0}$
denoted by $\omega({\RVI}_{0})$.
Let $h\in\omega({\RVI}_{0})$ and define
\begin{equation}\label{E-attr2}
f(t,x) \df -\partial_{t}\overline\Phi_{t}[h](x)
+\Lg^{v^{*}}\bigl(\overline\Phi_{t}[h](x) - V^{*}(x)\bigr)\,.
\end{equation}
Then $f(t,x)\ge0$ for all $(t,x)$,
and by applying It\^o's formula to \eqref{E-attr2}, we obtain
\begin{equation}\label{E-attr3}
\overline\Phi_{t}[h](x) - V^{*}(x) -
\Exp^{v^{*}}_{x}[h(X_{t})-V^{*}(X_{t})]
= - \Exp^{v^{*}}_{x}\left[\int_{0}^{t} f(t-s,X_{s})\,\D{s}\right]\,.
\end{equation}
Integrating \eqref{E-attr3} with respect to the invariant distribution
$\imeas_{v^{*}}$ we obtain
\begin{equation}\label{E-attr4}
\int_{\RR^{d}} \bigl(\overline\Phi_{t}[h](x)-h(x)\bigr)\,\imeas_{v^{*}}(\D{x})
= -\int_{0}^{t}\int_{\RR^{d}} f(t-s,x)\,\imeas_{v^{*}}(\D{x})\,\D{s}\qquad
\forall t\ge0\,.
\end{equation}
Since the term on the left-hand-side of \eqref{E-attr4}
equals $0$, as we argued above,
it follows that $f(t,x)=0$, $(t,x)-a.e.$, which in turn implies that
\begin{equation*}
\lim_{t\to\infty} \;\overline\Phi_{t}[h](x) = V^{*}(x) -
\int_{\RR^{d}} \bigl(V^{*}(x)-h(x)\bigr)\,\imeas_{v^{*}}(\D{x})\,.
\end{equation*}
It follows that $\omega({\RVI}_{0})\subset\Equil\cap\cG_{0}$
and since the map in \eqref{E-attr1} is nonincreasing, it is straightforward
to verify that $\omega({\RVI}_{0})$ must be a singleton.
\qquad
\end{proof}

We also have the following result which does not require
Assumption~\ref{A3.2}.

\begin{corollary}
Suppose ${\RVI}_{0}\in\Cc^{2}(\RR^{d})$ is
such that ${\RVI}_{0}-V^{*}$ is bounded.
Then $\overline\Phi_{t}[{\RVI}_{0}]$ converges as $t\to\infty$ to a point in $\Equil$.
\end{corollary}

\begin{proof}
By \eqref{E-bound}, under the hypothesis, 
$x\mapsto{\VI}(t,x)-V^{*}(x)$ is bounded uniformly in $t$.
Thus the result follows as in the proof of Theorem~\ref{T3.10}.
\qquad
\end{proof}

\section{Growth Estimates for Solutions of the Value Iteration}\label{S-growth}

Most of the results of this section do not require Assumption~\ref{A3.2}.
It is only need for Lemma~\ref{L-inner} and Corollary~\ref{C-Harnack}.
Throughout this section and also in \S\ref{S-main} a solution
${\VI}$ (${\RVI}$) always refers to the canonical solution of the
VI (RVI)
without further mention (see Definition~\ref{D-canonical}).

\begin{lemma}\label{L-asympt}
Suppose ${\RVI}_{0}\in\order_{V^{*}}$.
Then
\begin{equation*}
\frac{1}{t}\;{\VI}(t,x)\xrightarrow[t\to\infty]{}0\,.
\end{equation*}
\end{lemma}

\begin{proof}
Since $\norm{{\RVI}_{0}}_{V^{*}}<\infty$ it follows that
$\frac{1}{t}\Exp_{x}^{v^{*}}[{\RVI}_{0}(X_{t})]\to 0$ as $t\to\infty$
(see \cite[Lemma~3.7.2~(ii)]{book}),
and so we have
\begin{align*}
0 &\le \liminf_{t\to\infty}\;\frac{1}{t}\;
\inf_{U\in\Uadm}\;\Exp_{x}^{U}
\left[\int_{0}^{t}\Or(X_{s},U_{s})\,\D{s}
+ {\RVI}_{0}(X_{t})\right]\\[5pt]
&= \liminf_{t\to\infty}\;\frac{{\VI}(t,x)}{t}
\le \limsup_{t\to\infty}\;\frac{{\VI}(t,x)}{t}\\[5pt]
&\le
\limsup_{t\to\infty}\;
\frac{1}{t}\;\Exp_{x}^{v^{*}}
\left[\int_{0}^{t}
\Or\bigl(X_{s},v^{*}(X_{s})\bigr)\,\D{s}+{\RVI}_{0}(X_{t})\right]\\[5pt]
&=0\,.
\end{align*}
The first inequality above uses the fact that ${\RVI}_{0}$ is bounded below
and that $\varrho$ is the optimal ergodic cost.
\qquad
\end{proof}

\medskip
\begin{lemma}\label{L-nice}
Provided $\norm{{\RVI}_{0}}_{\infty}<\infty$,
it holds that for all $t\ge0$
\begin{equation*}
{\VI}(t-\tau,x) - {\VI}(t,x)\le \varrho\,\tau + \osc_{\RR^{d}}\;{\RVI}_{0}
\qquad \forall x\in\RR^{d}\,,\quad\forall \tau\in[0,t]\,.
\end{equation*}
\end{lemma}

{\em Proof}.
We have
\begin{align*}
{\VI}(t-\tau,x) - {\VI}(t,x) &= \inf_{U\in\Uadm}\;\Exp^{U}_{x}
\left[\int_{0}^{t-\tau}\Or(X_{s},U_{s})\,\D{s}+{\RVI}_{0}(X_{t-\tau})\right]
\nonumber\\[5pt]
&\mspace{100mu} - \inf_{U\in\Uadm}\;\Exp^{U}_{x}
\left[\int_{0}^{t}\Or(X_{s},U_{s})\,\D{s}+{\RVI}_{0}(X_{t})\right]
\nonumber\\[5pt]
&\le - \inf_{U\in\Uadm}\;\Exp^{U}_{x}\left[{\RVI}_{0}(X_{t})-{\RVI}_{0}(X_{t-\tau})
+ \int_{t-\tau}^{t}\Or(X_{s},U_{s})\,\D{s}\right]
\nonumber\\[5pt]
&\le \varrho\,\tau + \osc_{\RR^{d}}\;{\RVI}_{0}\,.\qquad\endproof
\end{align*}

\begin{definition}
We define:
\begin{equation*}
\cK\df\bigl\{x\in\RR^{d}: \min_{u\in\Act}\;r(x,u)\le \varrho\bigr\}\,.
\end{equation*}
Let $\Bo$ be some open bounded ball containing $\cK$ and
define $\tc\df\uptau(\Bo^{c})$.
Also let $\delta_{0}>0$ be such that $r(x,u)\ge\varrho+\delta_{0}$
on $\Bo^{c}$.
\end{definition}

\smallskip
\begin{lemma}
Suppose ${\RVI}_{0}\in\order_{V^{*}}$.
Then it holds that
\begin{equation}\label{E-hit-1}
{\VI}(t,x) \le\Exp_{x}^{v^{*}}
\left[\int_{0}^{\tc\wedge{t}}\Or\bigl(X_{s},v^{*}(X_{s})\bigr)\,\D{s}
+{\VI}(t-\tc\wedge{t},X_{\tc\wedge{t}})\right]\,,
\end{equation}
and
\begin{equation}\label{E-hit-2}
{\VI}(t,x) \ge\Exp_{x}^{\Hv^{t}}
\left[\int_{0}^{\tc\wedge{t}}\Or\bigl(X_{s},\Hv^{t}_{s}(X_{s})\bigr)\,\D{s}
+{\VI}(t-\tc\wedge{t},X_{\tc\wedge{t}})\right]
\end{equation}
for all $x\in \Bo^{c}$.
\end{lemma}

\begin{proof}
Let $B_{R}$ be any ball that contains $\overline{B}_{0}$ and
for $n\in\NN$, let $\uptau_{n}$ denote the first exit
time from $B_{nR}$.
Using Dynkin's formula on \eqref{EVI}, we obtain
\begin{equation}\label{E-hit-3}
{\VI}(t,x) =\inf_{U\in\Uadm}\;\Exp_{x}^{U}
\left[\int_{0}^{\tc\wedge\uptau_{n}\wedge{t}}\Or(X_{s},U_{s})\,\D{s}
+{\VI}(t-\tc\wedge\uptau_{n}\wedge{t},X_{\tc\wedge\uptau_{n}\wedge{t}})\right]
\end{equation}
for $x\in \Bo^{c}$.
By \eqref{E-hit-3} we have
\begin{align}\label{E-hit-4}
{\VI}(t,x) &\le \Exp_{x}^{v^{*}}
\left[\int_{0}^{\tc\wedge\uptau_{n}\wedge{t}}r\bigl(X_{s},v^{*}(X_{s})\bigr)
\,\D{s}\right]
- \varrho\,\Exp_{x}^{v^{*}}[\tc\wedge\uptau_{n}\wedge{t}]
\\[5pt]\nonumber
&\mspace{250mu}
+ \Exp_{x}^{v^{*}}\bigl[
{\VI}(t-\tc\wedge\uptau_{n}\wedge{t},X_{\tc\wedge\uptau_{n}\wedge{t}})\bigr]\,.
\end{align}
We use the expansion
\begin{align*}
\Exp_{x}^{v^{*}}\bigl[
{\VI}(t-\tc\wedge\uptau_{n}\wedge{t},X_{\tc\wedge\uptau_{n}\wedge{t}})\bigr]
&=
\Exp_{x}^{v^{*}}\bigl[
{\VI}(t-\tc\wedge{t},X_{\tc\wedge{t}})\,
\Ind\{\uptau_{n}>\tc\wedge{t}\}\bigr]
\nonumber\\[5pt]
&\mspace{80mu}
+ \Exp_{x}^{v^{*}}\bigl[{\VI}(t-\uptau_{n},X_{\uptau_{n}})\,
\Ind\{\uptau_{n}\le\tc\wedge{t}\}\bigr]\,.
\end{align*}
By \eqref{E-Vest} and the fact that,
as shown in \cite[Corollary~3.7.3]{book},
\begin{equation*}
\Exp_{x}^{v^{*}}\bigl[V^{*}(X_{\uptau_{n}})\,\Ind\{\uptau_{n}\le t\}\bigr]
\xrightarrow[n\to\infty]{}0\,,
\end{equation*}
we obtain
\begin{equation*}
\Exp_{x}^{v^{*}}\bigl[{\VI}(t-\uptau_{n},X_{\uptau_{n}})\,
\Ind\{\uptau_{n}\le\tc\wedge{t}\}\bigr]
\xrightarrow[n\to\infty]{}0\,.
\end{equation*}
Therefore by taking limits as $n\to\infty$ in \eqref{E-hit-4} and also using
monotone convergence for the first two terms on the r.h.s., we obtain
\eqref{E-hit-1}.

To obtain a lower bound we start from
\begin{equation}\label{E-hit-5}
{\VI}(t,x) =\Exp_{x}^{\Hv^{t}}\left[
\int_{0}^{\tc\wedge\uptau_{n}\wedge{t}}\Or\bigl(X_{s},\Hv^{t}_{s}(X_{s})\bigr)
\,\D{s}
+{\VI}(t-\tc\wedge\uptau_{n}\wedge{t},X_{\tc\wedge\uptau_{n}\wedge{t}})\right]\,.
\end{equation}
Since for any fixed $t$ the functions $\bigl\{{\VI}(t-s,x) : s\le t\bigr\}$
are uniformly bounded below, taking limits in \eqref{E-hit-5} as $n\to\infty$, we
obtain \eqref{E-hit-2}.
\qquad
\end{proof}

\smallskip
\begin{lemma}\label{L4.5}
Suppose ${\RVI}_{0}\in\order_{V^{*}}$.
Then for any $t>0$ we have
\begin{equation*}
{\VI}(t,x) > \min\;\biggl(\min_{[0,t]\times\overline{B}_{0}} \;{\VI}\,,
~\min_{\RR^{d}}\;{\RVI}_{0}\Bigr)\qquad \forall x\in \Bo^{c}\,.
\end{equation*}
\end{lemma}

\begin{proof}
Let $x$ be any point
in the interior of $\Bo^{c}$.
By \eqref{E-hit-2} we have
\begin{align*}
{\VI}(t,x) &\ge\Exp_{x}^{\Hv^{t}}\left[
\int_{0}^{\tc\wedge{t}}\Or\bigl(X_{s},\Hv^{t}_{s}(X_{s})\bigr)\,\D{s}
+{\VI}(t-\tc\wedge{t},X_{\tc\wedge{t}})\right]
\\[5pt]
&\ge \delta_{0} \Exp_{x}^{\Hv^{t}}[\tc\wedge{t}]
+ \Prob_{x}^{\Hv^{t}}(\tc \le t)\,
\min_{[0,t]\times\overline{B}_{0}} \;{\VI}
+ \Prob_{x}^{\Hv^{t}}(\tc > t)\,
\min_{\RR^{d}}\;{\RVI}_{0}\,,
\end{align*}
and the result follows.
\qquad
\end{proof}

\begin{remark}\label{R-crucial}
By Lemma~\ref{L4.5}, if $\inf_{[0,\infty)\times\overline{B}_{0}} \;{\VI}>-\infty$,
then ${\VI}$ is bounded below on $[0,\infty)\times\RR^{d}$.
If this the case, the convergence of the VI and therefore also of the RVI
follows as in the proof of Theorem~\ref{T3.10}.
Therefore without loss of generality we assume for the remainder
of the paper that $\inf_{[0,\infty)\times\overline{B}_{0}} \;{\VI}=-\infty$.
By Lemma~\ref{L4.5}, this implies that there exists $T_{0}>0$ such that
\begin{equation*}
{\VI}(t,x) > \min_{[0,t]\times\overline{B}_{0}} \;{\VI}
\qquad \forall t\ge T_{0}\,,~\forall x\in \Bo^{c}\,.
\end{equation*}
\end{remark}

We use the parabolic Harnack inequality which we quote
in simplified form from the more general result in \cite[Theorem~4.1]{Gruber}
as follows:

\begin{theorem}[Parabolic Harnack]\label{T-Harnack}
Let $B_{2R}\subset\RR^{d}$ be an open ball,
and $\psi$ be a nonnegative caloric function, i.e., a nonnegative solution of
\begin{equation*}
\partial_{t}\psi(t,x) - a^{ij}(t,x)\, \partial_{ij}\psi(t,x)
+b^{i}(t,x)\,\partial_{i}\psi(t,x)=0\quad\text{on~} [0,T]\times B_{2R}\,,
\end{equation*}
with $a^{ij}(t,x)$ continuous in $x$ and uniformly nondegenerate on
$[0,T]\times B_{2R}$, and $a^{ij}$ and $b^{i}$ bounded on $[0,T]\times B_{2R}$.
Then for any $\tau\in\bigl(0,\nicefrac{T}{4}\bigr]$,
there exists a constant $C_{H}$ depending only on $R$, $\tau$, and the ellipticity
constant (and modulus of continuity) of $a^{ij}$ and the bound of
$a^{ij}$ and $b^{i}$ on $B_{2R}$, such that
\begin{equation*}
\max_{[T-3\tau,T-2\tau]\times B_{R}}\;\psi
\le C_{H}\;\min_{[T-\tau,T]\times B_{R}}\;\psi\,.
\end{equation*}
\end{theorem}

In the three lemmas that follow we apply Theorem~\ref{T-Harnack} with $\tau\equiv 1$
and $\Bo'= 2 \Bo$.

\begin{lemma}\label{L-Harn1}
There exists a constant $M_{0}$ such that
\begin{equation*}
\max_{[T-3,T-2]\times\Bo}\; {\VI}
- \min_{[0,T]\times\Bo}\;{\VI}\le M_{0}
+C_{H}\Bigl(\min_{[T-1,T]\times\Bo}\;{\VI}
-\min_{[0,T]\times\Bo}\;{\VI}\Bigr)
\qquad\forall~ T\ge T_{0}+4\,.
\end{equation*}
\end{lemma}

{\em Proof}.
Let $\psi_{T}(t,x)$ be the unique
solution in $\Sobl^{1,2,p}\bigl((0,T)\times\Bo'\bigr)\cap
\Cc\bigl([0,T]\times\overline{B}_{0}'\bigr)$ of
\begin{gather*}
\partial_{t}\psi_{T}(t,x) - a^{ij}(x)\, \partial_{ij}\psi_{T}(t,x)
-b^{i}\bigl(x,\Hv^{T}_{t}(x)\bigr)\,\partial_{i}\psi_{T}(t,x)
=0\qquad\text{on}\;\; [0,T]\times\Bo'\,,
\\[5pt]
\psi_{T}(t,x) = {\VI}(t,x)
\qquad\text{on}\;\;
\bigl([0,T]\times\partial \Bo'\bigr)
\cup \bigl(\{0\}\times \overline{B}_{0}'\bigr)\,.
\end{gather*}
Since $\overline{\psi}_{T} \df \psi_{T} - {\VI}$ satisfies
\begin{equation*}
\partial_{t}\overline{\psi}_{T}(t,x) - a^{ij}(x)\,
\partial_{ij}\overline{\psi}_{T}(t,x)
-b^{i}\bigl(x,\Hv^{T}_{t}(x)\bigr)\,\partial_{i}\overline{\psi}_{T}(t,x)
+ \Or\bigr(x,\Hv^{T}_{t}(x)\bigr)=0
\end{equation*}
on $[0,T]\times\Bo'$, and
\begin{equation*}
\overline{\psi}_{T}(t,x) = 0\qquad\text{on}\;\;
\bigl([0,T]\times\partial \Bo'\bigr)
\cup \bigl(\{0\}\times \overline{B}_{0}'\bigr)\,,
\end{equation*}
it follows that there exists a constant $\overline{M}_{0}$ which depends only on
$\Bo'$
(it is independent of $T$) such that
\begin{equation}\label{E-Aleks}
\sup_{[0,T]\times\Bo'}\;\babs{\overline{\psi}_{T}}\le \overline{M}_{0}
\qquad\forall T>0\,.
\end{equation}
Indeed this is so because with $\uptau(\Bo')$ denoting the first exit time
from $\Bo'$ and with $\Hv^{T}$ as defined in Definition~\ref{D-minimizer}
we have
\begin{align*}
\babs{\overline{\psi}_{T} (t,x)} &= \left|\,\Exp^{\Hv^{T}}_{x}
\left[\int_{0}^{(T-t)\wedge\uptau(\Bo')}
\Or\bigl(X_{s},\Hv^{T}_{T-t-s}(X_{s})\bigr)\,\D{s}\right]\right|
\nonumber\\[5pt]
&\le \sup_{U\in\Uadm}\;\Exp^{U}_{x}
\left[\int_{0}^{\uptau(\Bo')}\babs{\Or(X_{s},U_{s})}\,\D{s}\right]
\nonumber\\[5pt]
&\le \norm{\Or}_{\infty,\Bo'}\;
\sup_{x\in \Bo'}\;\sup_{U\in\Uadm}\;\Exp^{U}_{x}[\uptau(\Bo')] < \infty\,,
\end{align*}
since the mean exit time from $\Bo'$ is upper bounded by a constant uniformly
over all initial $x\in \Bo'$ and all controls $U\in\Uadm$ by the
weak maximum principle of Alexandroff.

Let $(\Hat{t},\Hat{x})$ be a point at which ${\VI}$ attains its
minimum on $[T-1,T]\times\Bo$.
By Lemma~\ref{L4.5} and Remark~\ref{R-crucial}
the function $(t,x)\mapsto \psi_{T}(t,x) - \min_{[0,T]\times\Bo}\;{\VI}$
is nonnegative on $[T-4,T]\times\Bo'$, and by Theorem~\ref{T-Harnack} we have
\begin{align}\label{E-Harn1-1}
\psi_{T}(t,x) - \min_{[0,T]\times\Bo}\;{\VI}
&\le C_{H}
\Bigl(\psi_{T}(\Hat{t},\Hat{x}) - \min_{[0,T]\times\Bo}\;{\VI}\Bigr)
\\[5pt]\nonumber
&\le C_{H}
\Bigl(\overline{\psi}_{T}(\Hat{t},\Hat{x}) 
+ \min_{[T-1,T]\times\Bo}\;{\VI} - \min_{[0,T]\times\Bo}\;{\VI}\Bigr)
\end{align}
for all $t\in[T-3, T-2]$ and $x\in \Bo$\,.
Expressing the left hand side of \eqref{E-Harn1-1} as
\begin{equation*}
{\VI}(t,x) - \min_{[0,T]\times\Bo}\;{\VI}
+\overline{\psi}_{T}(t,x)\,,
\end{equation*}
and using \eqref{E-Aleks}, Lemma~\ref{L-Harn1} follows with 
\begin{equation*}
M_{0}\df(C_{H}+1)\overline{M}_{0}\,.\qquad\endproof
\end{equation*}

\begin{lemma}\label{L-outer}
Provided ${\RVI}_{0}\in\Cc^{2}(\RR^{d})$ is nonnegative and bounded, we have
\begin{equation*}
{\VI}(t,x) - \max_{\partial{B}_{0}}\;{\VI}(t,\cdot\,)
\le 2\,\norm{{\RVI}_{0}}_{\infty} + \bigl(1 + \varrho\,\delta_{0}^{-1}\bigr) V^{*}(x)
\qquad \forall x\in \Bo^{c}\,.
\end{equation*}
\end{lemma}

{\em Proof}.
By Lemma~\ref{L-nice}
\begin{equation}\label{E-outer-a}
{\VI}(t-\tau,x) \le {\VI}(t,x) +\varrho\,\tau + \osc_{\RR^{d}}\;{\RVI}_{0}
\qquad \forall x\in \Bo\,,\quad 0\le\tau\le t\,.
\end{equation}
Therefore by \eqref{E-hit-1} and \eqref{E-outer-a}, using the fact
that $\Or\ge0$ on $\Bo^{c}$, we obtain
\begin{align}\label{E-outer-1}
{\VI}(t,x) &\le\Exp_{x}^{v^{*}}
\left[\int_{0}^{\tc\wedge{t}}
\Or\bigl(X_{s},v^{*}(X_{s})\bigr)\,\D{s}
+{\VI}(t-\tc\wedge{t},X_{\tc\wedge{t}})\right]
\\[5pt]
&\le \Exp_{x}^{v^{*}}
\left[\int_{0}^{\tc}
\Or\bigl(X_{s},v^{*}(X_{s})\bigr)\,\D{s}\right]
+ \Exp_{x}^{v^{*}}\bigl[{\VI}(t-\tc,X_{\tc})\;\Ind\{\tc \le t\}\bigr]
\nonumber\\[5pt]
&\mspace{300mu}+\Exp_{x}^{v^{*}}\bigl[{\RVI}_{0}(X_{t})\;\Ind\{\tc > t\}\bigr]
\nonumber\\[5pt]
&\le V^{*}(x) + \Exp_{x}^{v^{*}}\bigl[{\VI}(t,X_{\tc})\,\Ind\{\tc \le t\}\bigr]
\nonumber\\[5pt]
&\mspace{200mu}+\varrho \Exp_{x}^{v^{*}}\bigl[\tc\,\Ind\{\tc \le t\}\bigr]
+ \osc_{\RR^{d}}\;{\RVI}_{0} + \norm{{\RVI}_{0}}_{\infty}
\nonumber\\[5pt]
&\le V^{*}(x) + \Prob_{x}^{v^{*}}\bigl(\{\tc \le t\}\bigr)\,
\Bigl(\max_{\partial\Bo}\;{\VI}(t,\cdot\,)\Bigr)
\nonumber\\[5pt]\nonumber
&\mspace{250mu} +\varrho \Exp_{x}^{v^{*}}\bigl[\tc\,\Ind\{\tc \le t\}\bigr]
+ 2\,\norm{{\RVI}_{0}}_{\infty}\,,
\end{align}
for $x\in \Bo^{c}$.
Since $-{\VI}(t,x)\le \varrho\, t$, we have
\begin{align}\label{E-outer-2}
- \Prob_{x}^{v^{*}}\bigl(\{\tc > t\}\bigr)\,
\Bigl(\max_{\partial\Bo}\;{\VI}(t,\cdot\,)\Bigr)
&\le \varrho\Prob_{x}^{v^{*}}\bigl(\{\tc > t\}\bigr)\,t
\\[5pt]\nonumber
&\le \varrho \Exp_{x}^{v^{*}}\bigl[\tc\,\Ind\{\tc > t\}\bigr]\,.
\end{align}
Hence subtracting $\max_{\partial\Bo}\;{\VI}(t,\cdot\,)$ from both
sides of \eqref{E-outer-1} and using \eqref{E-outer-2} together
with the estimate
$\Exp_{x}^{v^{*}}[\tc]\le \delta_{0}^{-1}V^{*}(x)$, we obtain
\begin{equation*}
{\VI}(t,x) - \max_{\partial\Bo}\;{\VI}(t,\cdot\,)
\le V^{*}(x) 
+\varrho\,\delta_{0}^{-1}V^{*}(x) + 2\,\norm{{\RVI}_{0}}_{\infty}\,.\qquad\endproof
\end{equation*}

We define the set $\cT\subset\RR_{+}$ by
\begin{equation*}
\cT \df \Bigl\{ t\ge T_{0}+4 :
\min_{[t-1,t]\times\Bo}\;{\VI} = \min_{[0,t]\times\Bo}{\VI}\Bigr\}\,,
\end{equation*}
where $T_{0}$ is as in Remark~\ref{R-crucial}.
By Remark~\ref{R-crucial}, $\cT\ne\varnothing$.

\begin{lemma}\label{L-inner}
Let Assumption~\ref{A3.2} hold and
suppose that the initial condition ${\RVI}_{0}\in\Cc^{2}(\RR^{d})$
is nonnegative and bounded.
Then there exists a constant $C_{0}$ such that
\begin{equation*}
\osc_{\Bo}\;{\VI}(t,\cdot\,)
\le C_{0}\qquad \forall t\ge0\,.
\end{equation*}
\end{lemma}

\begin{proof}
Suppose $t\in\cT$.
Then, by Lemma~\ref{L-Harn1},
\begin{equation*}
\max_{\partial \Bo}\;{\VI}(t-2,\cdot\,) - \min_{[0,t]\times\Bo}\;{\VI} \le M_{0}\,.
\end{equation*}
Therefore, by Lemma~\ref{L-outer} we have
\begin{equation}\label{E-inner-1}
{\VI}(t-2,x) - \min_{[0,t]\times\Bo}\;{\VI}
\le M_{0} + 2\,\norm{{\RVI}_{0}}_{\infty} 
+ \bigl(1 + \varrho\,\delta_{0}^{-1}\bigr) V^{*}(x)
\qquad \forall (t,x)\in\cT\times\Bo^{c}\,.
\end{equation}
Next, fix any $t_{0}\in\cT$.
It suffices to prove the result for $t\ge t_{0}$
since it trivially holds for $t$ in the compact interval $[0,t_{0}]$.
Given $t\ge t_{0}$ let $\tau\df \sup\; \cT\cap [t_{0},t]$.
Note then that
\begin{equation}\label{E-inner-2}
\min_{[0,\tau]\times\Bo}\;{\VI}=\min_{[0,t]\times\Bo}\;{\VI}\,.
\end{equation}
By \eqref{E-inner-1}--\eqref{E-inner-2}, and since $V^{*}$ is nonnegative, we obtain
\begin{align}\label{E-inner-3}
\sup_{x\in \Bo}\;{\VI}(t,x)
&\le \sup_{x\in \Bo}\; \Exp_{x}^{v^{*}}
\left[\int_{0}^{t-\tau+2} \Or\bigl(X_{s},v^{*}(X_{s})\bigr)\,\D{s}
+ {\VI}(\tau-2,X_{t-\tau+2})\right]
\\[5pt]
&\le \sup_{x\in \Bo}\; \Exp_{x}^{v^{*}}
\left[\int_{0}^{t-\tau+2} \Or\bigl(X_{s},v^{*}(X_{s})\bigr)\,\D{s}
+ V^{*}(X_{t-\tau+2})\right]
\nonumber\\[5pt]
&\mspace{150mu}
+ \sup_{x\in \Bo}\; \Exp_{x}^{v^{*}}\bigl[{\VI}(\tau-2,X_{t-\tau+2})\bigr]
\nonumber\\[5pt]\nonumber
&\le \norm{V^{*}}_{\infty,\Bo} + \min_{[0,t]\times\Bo}\;{\VI} + M_{0}
+ 2\,\norm{{\RVI}_{0}}_{\infty} + \varrho\,\delta_{0}^{-1}K_{0}\,,
\end{align}
with
\begin{equation*}
K_{0}\df\sup_{t\ge 0}\; \sup_{x\in \Bo}\;\Exp_{x}^{v^{*}}
\bigl[V^{*}(X_{t})\bigr]\,.
\end{equation*}
By Lemma~\ref{L3.4}, $K_{0}$ is finite.
Since
\begin{equation*}
\osc_{\Bo}\;{\VI}(t,\cdot\,)\le\sup_{x\in \Bo}\;{\VI}(t,x)
-\min_{[0,t]\times\Bo}\;{\VI}\,,
\end{equation*}
and $t\ge t_{0}$ was arbitrary,
the result follows for all $t\ge t_{0}$ by \eqref{E-inner-3}.
\qquad
\end{proof}

The following corollary
now follows by Theorem~\ref{T-Harnack} and Lemma~\ref{L-inner}.

\begin{corollary}\label{C-Harnack}
Under the hypotheses of Lemma~\ref{L-inner},
for any $\tau>0$ there exists a constant $\widehat{C}(\tau)$
such that
\begin{equation*}
\osc_{[n\tau,(n+1)\tau]\times\Bo}\;{\VI}\le \widehat{C}(\tau)\qquad\forall n\in\NN\,.
\end{equation*}
\end{corollary}

\section{Convergence of the Relative Value Iteration}\label{S-main}

We define the set $\cT_{0}\subset\RR_{+}$ by
\begin{equation*}
\cT_{0} \df \bigl\{ t\in\RR_{+} : {\VI}(t,0) \le {\VI}(t',0)
\quad\forall t'\le t\bigr\}\,.
\end{equation*}

In the next lemma we use the variable
\begin{equation*}
\varPsi(t,x) \df {\VI}(t,x)-{\VI}(t,0)\,.
\end{equation*}

\begin{lemma}\label{L5.1}
Let Assumption~\ref{A3.2} hold and also
suppose that the initial condition ${\RVI}_{0}\in\Cc^{2}(\RR^{d})$
is nonnegative and bounded.
Then
\begin{subequations}
\begin{gather}
\varPsi(t,x) \le  C_{0} + 2\,\norm{{\RVI}_{0}}_{\infty} 
+ \bigl(1 + \varrho\,\delta_{0}^{-1}\bigr) V^{*}(x) \qquad
\forall (t,x)\in\RR_{+}\times\RR^{d}\,,
\label{E5.1a}
\intertext{and there exists a constant $\widehat{M}_{0}$ such that}
{\VI}(t,0) - {\VI}(t',0)\le \widehat{M}_{0}\qquad \forall t\ge t'\,.
\label{E5.1b}
\end{gather}
\end{subequations}
\end{lemma}

\begin{proof}
The estimate in \eqref{E5.1a} follows by Lemmas~\ref{L-outer} and \ref{L-inner}.
To show \eqref{E5.1b} note that
\begin{equation}\label{EL5.1a}
{\VI}(t,0) - {\VI}(t',0) \le {\VI}(t,0) - \min_{s\in[0,t]}\;{\VI}(s,0)\\
\qquad \forall t\in[0,t]\,.
\end{equation}
Let $t^{*}\in\Argmin_{s\in[0,t]}\;{\VI}(s,0)$ and define $T\df t-t^{*}$.
Clearly, $t^{*}= t-T\in\cT_{0}$.
We have
\begin{align}\label{E5.2}
{\VI}(t,0) - {\VI}(t-T,0)&\le \Exp_{0}^{v^{*}}
\left[\int_{0}^{T} \Or\bigl(X_{s},v^{*}(X_{s})\bigr)\,\D{s}
+ {\VI}(t-T,X_{T})\right]
\\[5pt]
&\mspace{300mu}-{\VI}(t-T,0) 
\nonumber\\[5pt]
&= \Exp_{0}^{v^{*}}
\left[\int_{0}^{T} \Or\bigl(X_{s},v^{*}(X_{s})\bigr)\,\D{s}
+ \varPsi(t-T,X_{T})\right]
\nonumber\\[5pt]
&= V^{*}(0) - \Exp_{0}^{v^{*}}[V^{*}(X_{T})]
+\Exp_{0}^{v^{*}}\bigl[\varPsi(t-T,X_{T})\bigr]
\nonumber\\[5pt]\nonumber
&\le V^{*}(0) + C_{0} + 2\,\norm{{\RVI}_{0}}_{\infty}
+ \varrho\,\delta_{0}^{-1}\,\Exp_{0}^{v^{*}}\bigl[V^{*}(X_{T})\bigr]\,,
\end{align}
where the last inequality follows by \eqref{E5.1a}.
However, by Lemma~\ref{L3.4} there exists a constant $\widetilde{M}_{0}$ such
that
\begin{equation*}
\sup_{T\ge0}\; \Exp_{0}^{v^{*}}[V^{*}(X_{T})] \le \widetilde{M}_{0}\,.
\end{equation*}
It then follows by \eqref{E5.2} that
${\VI}(t,0) - {\VI}(t-T,0)$ is bounded above by a constant independent
of $t$ and $T$.
The result then follows by \eqref{EL5.1a}.
\qquad
\end{proof}

\smallskip
\begin{lemma}\label{L5.2}
Under the hypotheses of Lemma~\ref{L5.1} there exists a constant $k_{0}>0$ such that
\begin{equation*}
\Exp_{x}^{\Hv^{t}}[\tc\wedge{t}]
\le k_{0} + 2\,\delta_{0}^{-1}\bigl(1 + \varrho\,\delta_{0}^{-1}\bigr) V^{*}(x)
\qquad \forall x\in \Bo^{c}\,.
\end{equation*}
\end{lemma}

\begin{proof}
Subtracting ${\VI}(t,0)$ from both sides of \eqref{E-hit-2}, we obtain
\begin{align*}
\varPsi(t,x)
&\ge\Exp_{x}^{\Hv^{t}}
\biggl[\int_{0}^{\tc\wedge{t}}\Or\bigl(X_{s},\Hv^{t}_{s}(X_{s})\bigr)\,\D{s}
+\varPsi(t-\tc\wedge{t},X_{\tc\wedge{t}})\,\Ind\{\tc\le t\}
+ {\RVI}_{0}(X_{t})\;\Ind\{\tc > t\}
\nonumber\\[5pt]
&\mspace{120mu}- {\VI}(t,0)\,\Ind\{\tc>t\}
+\bigl({\VI}(t-\tc\wedge{t},0)-{\VI}(t,0)\bigr)\,\Ind\{\tc\le t\}\biggr]\,.
\end{align*}
We discard the nonnegative term ${\RVI}_{0}(X_{t})\;\Ind\{\tc > t\}$,
and we use Lemma~\ref{L-inner} and \eqref{E5.1b}
to write the above inequality as
\begin{align}\label{E5.3}
\varPsi(t,x) &\ge \Exp_{x}^{\Hv^{t}}
\biggl[\int_{0}^{\tc\wedge{t}}
\Or\bigl(X_{s},\Hv^{t}_{s}(X_{s})\bigr)\,\D{s}\biggr]
- \sup_{0\le s\le t}\;\norm{\varPsi(s,\cdot\,)}_{\infty,\Bo}
\\[5pt]\nonumber
&\mspace{30mu}
-\Exp_{x}^{\Hv^{t}}\bigl[{\VI}(t,0)\,\Ind\{\tc>t\}\bigr]
+\Exp_{x}^{\Hv^{t}}
\bigl[\bigl({\VI}(t-\tc\wedge{t},0)-{\VI}(t,0)\bigr)\,\Ind\{\tc\le t\}\bigr]
\\[5pt]\nonumber
&\ge
\Exp_{x}^{\Hv^{t}}
\biggl[\int_{0}^{\tc\wedge{t}}
\Or\bigl(X_{s},\Hv^{t}_{s}(X_{s})\bigr)\,\D{s}\biggr] - C_{0}
-{\VI}(t,0)\Prob_{x}^{\Hv^{t}}\bigl(\{\tc>t\}\bigr) - \widehat{M}_{0}\,.
\end{align}
By \eqref{E5.1a} and \eqref{E5.3} we obtain
\begin{align*}
C_{0} + 2\,\norm{{\RVI}_{0}}_{\infty}
+ \bigl(1 + \varrho\,\delta_{0}^{-1}\bigr) V^{*}(x) &\ge
\delta_{0} \Exp_{x}^{\Hv^{t}}[\tc\wedge{t}] 
-{\VI}(t,0)\Prob_{x}^{\Hv^{t}}\bigl(\{\tc>t\}\bigr) 
\\[5pt]\nonumber
&\mspace{150mu}- C_{0}- \widehat{M}_{0}
\\[5pt]\nonumber
&\ge
\left(\delta_{0} - \tfrac{{\VI}(t,0)}{t}\right)
\Exp_{x}^{\Hv^{t}}[\tc\wedge{t}]
 - C_{0}- \widehat{M}_{0}\,.
\end{align*}
The result then follows by Lemma~\ref{L-asympt}.
\qquad
\end{proof}

\smallskip
\begin{lemma}\label{L5.3}
Under the hypotheses of Lemma~\ref{L5.1},
\begin{equation*}
{\VI}(t,0)\Prob_{x}^{\Hv^{t}}\bigl(\{\tc>t\}\bigr)
\xrightarrow[t\to\infty]{} 0\,,
\end{equation*}
uniformly on $x$ in compact sets of $\RR^{d}$\,.
\end{lemma}

{\em Proof}.
By Lemma~\ref{L-asympt} and Lemma~\ref{L5.2} we have
\begin{equation*}
{\VI}(t,0)\Prob_{x}^{\Hv^{t}}\bigl(\{\tc>t\}\bigr)
\le \frac{{\VI}(t,0)}{t}\bigl(k_{0} + 2\,\delta_{0}^{-1}
\bigl(1 + \varrho\,\delta_{0}^{-1}\bigr) V^{*}(x)\bigr)
\xrightarrow[t\to\infty]{} 0
\qquad \forall x\in \Bo^{c}\,.\qquad\endproof
\end{equation*}

\smallskip
\begin{lemma}\label{L-Vbound}
Let Assumption~\ref{A3.2} hold and also
suppose the initial condition ${\RVI}_{0}\in\Cc^{2}(\RR^{d})$
is nonnegative and bounded.
Then the map $t\mapsto {\RVI}(t,0)$ is bounded on $[0,\infty)$,
and it holds that
\begin{equation*}
- \osc_{\RR^{d}}\;{\RVI}_{0}\le \liminf_{t\to\infty}\; {\RVI}(t,0)
\le \limsup_{t\to\infty}\; {\RVI}(t,0)
\le \widehat{M}_{0}+\varrho\,.
\end{equation*}
\end{lemma}

\begin{proof}
Define
\begin{equation*}
g(t)\df 
\inf_{U\in\Uadm}\;\Exp^{U}_{0}
\left[\int_{0}^{t}r(X_{s},U_{s})\,\D{s} + {\RVI}_{0}(X_{t})\right]\,.
\end{equation*}
By \eqref{E-SR} we have
\begin{align}\label{E-Vbound-1}
{\RVI}(t,0) &= g(t) - \int_{0}^{t} \E^{s-t} g(s)\,\D{s}
\\[5pt]\nonumber
&= \bigl(1-\E^{-t}\bigr)^{-1}
\int_{0}^{t} \E^{s-t}\bigl(g(t)-g(s)\bigr)\,\D{s}
\\[5pt]\nonumber
&\mspace{200mu}+\bigl(1-\E^{-t}\bigr)^{-1}\E^{-t}\int_{0}^{t} \E^{s-t} g(s)\,\D{s}\,,
\end{align}
for $t>0$.
By Lemma~\ref{L5.1}, $g(t)\le \widehat{M}_{0}+{\RVI}_{0}(0)+\varrho\, t$.
Therefore the second term on the right hand side of \eqref{E-Vbound-1} vanishes
as $t\to\infty$.
By Lemma~\ref{L-nice}, $g(t)-g(s)\ge -\osc_{\RR^{d}}\;{\RVI}_{0}$
for all $s\le t$.
Also, by Lemma~\ref{L5.1},
$g(t)-g(s)\le \widehat{M}_{0}+\varrho(t-s)$ for all $s\le t$.
Evaluating the first integral on  the right hand side of\eqref{E-Vbound-1}
we obtain the bound
\begin{equation}\label{E-Vbound-2}
-\osc_{\RR^{d}}\;{\RVI}_{0}\le
\int_{0}^{t} \E^{s-t}\bigl(g(t)-g(s)\bigr)\,\D{s}
\le \widehat{M}_{0}+\varrho\qquad \forall t>0\,.
\end{equation}
The result follows by \eqref{E-Vbound-1}--\eqref{E-Vbound-2}.
\qquad
\end{proof}

Combining Corollary~\ref{C-Harnack}, 
the boundedness of $t\mapsto {\RVI}(t,0)$ asserted in Lemma~\ref{L-Vbound},
and \eqref{E-ident}, it follows that $x\mapsto {\RVI}(t,x)$
is locally bounded in $\RR^{d}$, uniformly in $t\ge0$.
Recall Definition~\ref{D-flow}. 
The standard interior estimates of the solutions of \eqref{ERVI} provide us
with the following regularity result:

\begin{theorem}
Under the hypotheses of Lemma~\ref{L-Vbound}
the closure of the orbit $\{\Phi_{t}[{\RVI}_{0}]\,,\ t\in\RR_{+}\}$
is locally compact in $\Cc^{2}(\RR^{d})$.
\end{theorem}

\begin{proof}
By Corollary~\ref{C-Harnack} and Lemma~\ref{L-inner}, the oscillation
of ${\VI}$ is bounded on any cylinder $[n,n+1]\times B_{R}$, uniformly
over $n\in\NN$.
This together with Lemma~\ref{L-Vbound} imply that $\Phi_{t}[{\RVI}_{0}](x)$
is bounded on $(t,x)\in [n,n+1]\times B_{R}$, for any $R>0$, uniformly
in $n\in\NN$.
It follows that the derivatives
$\partial_{ij}\Phi_{t}[{\RVI}_{0}]$
are H\"older equicontinuous on every ball $B_{R}$ uniformly in $t$
\cite[Theorem~5.1]{Lady}.
The result follows.
\end{proof}

\smallskip
We now turn to the proof of our main result.

\smallskip
{\em Proof of Theorem~\ref{T-main}}.
Let $\{t_{n}\}$ be any diverging sequence and let $f$ be
any limit in in the topology of Markov controls
(see \cite[Section~2.4]{book}) of $\{\Hv^{t_{n}}\}$
along some subsequence of $\{t_{n}\}$ also denoted as $\{t_{n}\}$.
By Fatou's lemma and the stochastic representation of $V^{*}$ in Theorem~\ref{T3.1},
we have,
\begin{align}\label{ETm1}
\liminf_{n\to\infty}\;\Exp_{x}^{\Hv^{t_{n}}}
\biggl[\int_{0}^{\tc\wedge{t}_{n}}\Or\bigl(X_{s},\Hv^{t_{n}}_{s}(X_{s})\bigr)
\,\D{s}\biggr]
&\ge \Exp_{x}^{f}
\biggl[\int_{0}^{\tc}\Or\bigl(X_{s},f_{s}(X_{s})\bigr)\,\D{s}\biggr]
\\[5pt]
&\ge \inf_{v\in\Ussm}\;\Exp_{x}^{v}
\biggl[\int_{0}^{\tc}\Or\bigl(X_{s},v(X_{s})\bigr)\,\D{s}\biggr]
\nonumber\\[5pt]\nonumber
&\ge V^{*}(x) - \norm{V^{*}}_{\infty,\Bo}\qquad\forall x\in \Bo^{c}\,.
\end{align}
The second inequality in \eqref{ETm1} is due to the fact that
the infimum of
\begin{equation*}
\Exp_{x}^{U}\biggl[\int_{0}^{\tc}\Or\bigl(X_{s},U_{s})\bigr)\,\D{s}\biggr]
\end{equation*}
over all $U\in\Uadm$ is realized at some $v\in\Ussm$, while the
third inequality follows by \eqref{E-strepr}.
Therefore, by \eqref{E-ident}, \eqref{E5.3}, \eqref{ETm1} and
Lemmas~\ref{L5.3} and \ref{L-Vbound}
we have that
\begin{align}\label{ETm2}
\liminf_{t\to\infty}\;{\RVI}(t,x)
&= \liminf_{t\to\infty}\; \bigl(\varPsi(t,x)+{\RVI}(t,0)\bigr)
\\[5pt]\nonumber
&\ge
V^{*}(x)- \norm{V^{*}}_{\infty,\Bo} 
- C_{0}-\widehat{M}_{0}- \osc_{\RR^{d}}\;{\RVI}_{0}
\qquad \forall x\in\Bo^{c}\,.
\end{align}
Also, by \eqref{E5.1a} and Lemma~\ref{L5.3} we obtain
\begin{align}\label{ETm3}
\limsup_{t\to\infty}\;{\RVI}(t,x)
&= \limsup_{t\to\infty}\; \bigl(\varPsi(t,x)+{\RVI}(t,0)\bigr)
\varPsi(t,x)
\\[5pt]\nonumber
&\le  C_{0} + 2\,\norm{{\RVI}_{0}}_{\infty} 
+ \bigl(1 + \varrho\,\delta_{0}^{-1}\bigr) V^{*}(x) + \widehat{M}_{0}+\varrho
\end{align}
for all $(t,x)\in\RR_{+}\times\RR^{d}$.

Hence, by \eqref{ETm2}--\eqref{ETm3} if we select
\begin{equation*}
c= -\bigl(\norm{V^{*}}_{\infty,\Bo} 
+C_{0}+\widehat{M}_{0}+ \osc_{\RR^{d}}\;{\RVI}_{0})\,,
\end{equation*} 
then any $\omega$-limit point
of ${\RVI}(t,x)$ as $t\to\infty$ lies in $\cG_{c}$
(see Definition~\ref{D-flow}).
By Theorem~\ref{T3.10} if ${\RVI}_{0}\in \cG_{c}$,
then ${\RVI}(t,x)\to V^{*}(x)+\varrho$ as
$t\to\infty$.
Since the $\omega$-limit set of ${\RVI}_{0}$ is invariant and the only invariant
set in $\cG_{c}$ is the singleton $\{V^{*}-V^{*}(0)+\varrho\}$ the result follows.
\qquad
\endproof

\section{Concluding Remarks}\label{S-concl}
We have studied the relative value iteration algorithm
for an important class of ergodic control problems wherein instability
is possible, but is heavily penalized by the near-monotone structure of the
running cost.
The near-monotone cost structure plays a crucial role in the analysis
and the proof of stabilization of the quasilinear parabolic Cauchy initial
value problem that models the algorithm.

We would like to conjecture that the RVI converges starting from
any initial condition ${\RVI}_{0}\in\order_{V^{*}}$.
It is only the estimate in Lemma~\ref{L-nice} that restricts us
to consider bounded initial conditions only.
We want to mention here that a related such estimate can be
obtained as follows:
\begin{align*}
{\VI}(t,x) &= \inf_{U\in\Uadm}\;\Exp^{U}_{x}
\left[\int_{0}^{t}\Or(X_{s},U_{s})\,\D{s}
+{\RVI}_{0}(X_{t})\right]
\nonumber\\[5pt]
&= \inf_{U\in\Uadm}\;\Exp^{U}_{x}
\left[\int_{0}^{\tau}\Or(X_{s},U_{s})\,\D{s}
+{\VI}(t-\tau,X_{t-\tau})\right]
\nonumber\\[5pt]
&\ge -\varrho\,\tau + \min_{y\in\RR^{d}}\;{\VI}(t-\tau,y)\qquad
\forall \tau\in[0,t]\,,\quad\forall x\in\RR^{d}\,.
\end{align*}
In particular
\begin{equation*}
\min_{\RR^{d}}\;{\VI}(t-\tau,\cdot\,)
-\min_{\RR^{d}}\; {\VI}(t,\cdot\,)\le \varrho\,\tau
\qquad \forall \tau\in[0,t]\,,
\end{equation*}
and this estimate does not depend on the initial data ${\RVI}_{0}$.
This suggests that it is probably worthwhile studying the variation of the
RVI algorithm that results by replacing ${\RVI}(t,0)$ by
$\min_{\RR^{d}}\; {\RVI}(t,\cdot\,)$ in \eqref{E-RVI}.

Rate of convergence results and computational aspects of the algorithm
are open issues.

\newpage
\def\cprime{$'$} \def\cprime{$'$} \def\cprime{$'$}


\begin{thebibliography}{10}

\bibitem{Abounadi-01}
{\sc J.~Abounadi, D.~P. Bertsekas, and V.~S. Borkar}, {\em Learning algorithms
  for {M}arkov decision processes with average cost}, SIAM J. Control Optim.,
  40 (2001), pp.~681--698.

\bibitem{OneFest}
{\sc A.~Arapostathis}, {\em On the policy iteration algorithm for nondegenerate
  controlled diffusions under the ergodic criterion}, in Optimization, control,
  and applications of stochastic systems, Systems Control Found. Appl.,
  Birkh\"auser/Springer, New York, 2012, pp.~1--12.

\bibitem{RVI}
{\sc A.~Arapostathis and V.~S. Borkar}, {\em A relative value iteration
  algorithm for nondegenerate controlled diffusions}, SIAM J. Control Optim.,
  50 (2012), pp.~1886--1902.

\bibitem{book}
{\sc A.~Arapostathis, V.~S. Borkar, and M.~K. Ghosh}, {\em Ergodic control of
  diffusion processes}, vol.~143 of Encyclopedia of Mathematics and its
  Applications, Cambridge University Press, Cambridge, 2011.

\bibitem{arXiv-games}
{\sc A.~Arapostathis, V.~S. Borkar, and K.~Suresh Kumar}, {\em Relative value
  iteration for stochastic differential games}, arXiv:1210.8188 (2012).

\bibitem{Bogachev-01}
{\sc V.~I. Bogachev, N.~V. Krylov, and M.~R{\"o}ckner}, {\em On regularity of
  transition probabilities and invariant measures of singular diffusions under
  minimal conditions}, Comm. Partial Differential Equations, 26 (2001),
  pp.~2037--2080.

\bibitem{ChenMeyn-99}
{\sc R.-R. Chen and S.~Meyn}, {\em Value iteration and optimization of
  multiclass queueing networks}, Queueing Systems Theory Appl., 32 (1999),
  pp.~65--97.

\bibitem{Eidelman-00}
{\sc S.~Eidelman, S.~Kamin, and F.~Porper}, {\em Uniqueness of solutions of the
  {C}auchy problem for parabolic equations degenerating at infinity}, Asymptot.
  Anal., 22 (2000), pp.~349--358.

\bibitem{Gruber}
{\sc M.~Gruber}, {\em Harnack inequalities for solutions of general second
  order parabolic equations and estimates of their {H}\"older constants}, Math.
  Z., 185 (1984), pp.~23--43.

\bibitem{Gyongy-96}
{\sc I.~Gy{\"o}ngy and N.~Krylov}, {\em Existence of strong solutions for
  {I}t\^o's stochastic equations via approximations}, Probab. Theory Related
  Fields, 105 (1996), pp.~143--158.

\bibitem{Hasm-60}
{\sc R.~Z. Has$'$minski\u{\i}}, {\em Ergodic properties of recurrent diffusion
  processes and stabilization of the solution of the {C}auchy problem for
  parabolic equations}, Theory Probab. Appl., 5 (1960), pp.~179--196.

\bibitem{Krylov}
{\sc N.~V. Krylov}, {\em Controlled diffusion processes}, vol.~14 of
  Applications of Mathematics, Springer-Verlag, New York, 1980.

\bibitem{Krylov-08}
\leavevmode\vrule height 2pt depth -1.6pt width 23pt, {\em Lectures on elliptic
  and parabolic equations in {S}obolev spaces}, vol.~96 of Graduate Studies in
  Mathematics, American Mathematical Society, Providence, RI, 2008.

\bibitem{Lady}
{\sc O.~A. Lady{\v{z}}enskaja, V.~A. Solonnikov, and N.~N. Ural{\cprime}ceva},
  {\em Linear and quasilinear equations of parabolic type}, Translated from the
  Russian by S. Smith. Translations of Mathematical Monographs, Vol. 23,
  American Mathematical Society, Providence, R.I., 1967.

\bibitem{Locherbach-11}
{\sc E.~L{\"o}cherbach, D.~Loukianova, and O.~Loukianov}, {\em Polynomial
  bounds in the ergodic theorem for one-dimensional diffusions and
  integrability of hitting times}, Ann. Inst. Henri Poincar\'e Probab. Stat.,
  47 (2011), pp.~425--449.

\bibitem{Meyn-09}
{\sc S.~Meyn and R.~L. Tweedie}, {\em Markov chains and stochastic stability},
  Cambridge University Press, Cambridge, second~ed., 2009.

\bibitem{Meyn}
{\sc S.~P. Meyn}, {\em The policy iteration algorithm for average reward
  {M}arkov decision processes with general state space}, IEEE Trans. Automat.
  Control, 42 (1997), pp.~1663--1680.

\bibitem{MT-III}
{\sc S.~P. Meyn and R.~L. Tweedie}, {\em Stability of {M}arkovian processes.
  {III}. {F}oster-{L}yapunov criteria for continuous-time processes}, Adv. in
  Appl. Probab., 25 (1993), pp.~518--548.

\bibitem{Stannat-99}
{\sc W.~Stannat}, {\em ({N}onsymmetric) {D}irichlet operators on {$L\sp 1$}:
  existence, uniqueness and associated {M}arkov processes}, Ann. Scuola Norm.
  Sup. Pisa Cl. Sci. (4), 28 (1999), pp.~99--140.

\bibitem{White-63}
{\sc D.~J. White}, {\em Dynamic programming, {M}arkov chains, and the method of
  successive approximations}, J. Math. Anal. Appl., 6 (1963), pp.~373--376.

\end{thebibliography}
\end{document}